\documentclass{article}
\usepackage{graphicx} 
\usepackage{subcaption}
\usepackage{amsmath}
\usepackage{amssymb}
\usepackage{amsthm}
\usepackage{xcolor}
\usepackage{tabularx}
\usepackage{geometry}
\usepackage{hyperref}
\usepackage{doi} 
\usepackage{multirow}

\usepackage{authblk}

\usepackage[most]{tcolorbox}

\usepackage{float}

\usepackage{url} 

\usepackage{tikz}

\newtheorem{problem}{Problem}

\newcommand{\real}{\mathbb{R}}
\newcommand{\order}{\mathcal{O}}
\newcommand{\lc}{\mathfrak{lc}}
\newcommand{\cc}{\mathfrak{cc}}
\newcommand{\stab}{\mathcal{S}}
\newcommand{\Rinf}{R_\infty}

\newcommand*\circled[1]{\tikz[baseline=(char.base)]{
            \node[shape=circle,draw,inner sep=2pt] (char) {#1};}}

\newtheorem{thm}{Theorem}[section]

\newtheorem{cor}[thm]{Corollary}
\newtheorem{prop}[thm]{Proposition}
\newtheorem{dfn}[thm]{Definition}
\newtheorem{rem}[thm]{Remark}

\newtheorem{ex}[thm]{Example}
\newtheorem{algo}[thm]{Algorithm}

\title{Optimized linear multistep methods for stiff problems}

\author[1]{Tyler Jensen\thanks{Corresponding author: tsj55605@gmail.com}}
\author[2]{Lajos L\'oczi\thanks{Lajos.Loczi@kaust.edu.sa, ORCID: 0000-0002-7999-5658}}
\author[3]{David I.~Ketcheson\thanks{David.Ketcheson@kaust.edu.sa, ORCID: 0000-0002-1212-126X}}

\affil[1]{Department of Physics and Astronomy, Brigham Young University, Provo, United States}
\affil[2,3]{CEMSE Division, King Abdullah University of Science and Technology (KAUST), Thuwal, Saudi Arabia}

\begin{document}

\maketitle

\begin{abstract}
We design and test optimized linear multistep methods
for stiff problems.  Each method is A-stable, implicit, and second-order
accurate.  We focus on achieving the smallest possible error constant while
also ensuring that stiff modes are damped.
We recover some known optimal L-stable three- and four-step methods and
prove their global optimality, and provide a new five-step method with
essentially optimal error constant. We also study methods that are not
L-stable but provide a prescribed rate of damping in the infinitely stiff
limit; in particular, we prove that a known two-step family is optimal.
Variable-step-size formulae are also provided for the L-stable methods.
Numerical tests demonstrate that the improved error constants lead to
improved accuracy.
\end{abstract}

\noindent \textbf{Keywords:} stiff ordinary differential equations, A-stable linear multistep methods, optimal error constants \\

\noindent \textbf{Mathematics Subject Classification:} 65L06, 65L04, 65L20, 65L05

\subsection*{Funding Declaration}

This work was supported by funding from King Abdullah University of Science and Technology (KAUST), including support through the Visiting Student Research Program (VSRP).

\section{Background and motivation}
The efficient solution of stiff ordinary differential equations (ODEs)
typically employs the use of A-stable and L-stable numerical methods.  The effectiveness of such methods is demonstrated by, for example, the very widespread use of the 2-step $2^\text{nd}$-order A-stable and L-stable backward difference formula (BDF2)
\begin{equation}\label{BDF2LMMformula}
        y_{n+2}-\frac{4}{3}y_{n+1}+\frac{1}{3}y_{n}=\frac{2h}{3}f(y_{n+2}).
\end{equation}
Unfortunately, as stated in Dahlquist's second barrier \cite[Section V.1]{wanner1996solving}, no linear multistep method (LMM) can be A-stable and higher than $2^\text{nd}$-order accurate.  Thus,  the applicability of BDF methods of order $p\ge 3$ is more limited.
One alternative is the use of multistage (e.g., Runge--Kutta) methods, and these are indeed commonly utilized. 
However, for stiff problems, they can suffer from order reduction and they may require the solution of very large systems of algebraic equations. Consequently, in this work we focus on A-stable (hence necessarily implicit) LMMs of order $p=2$. 


It is worth asking
whether one can improve on the BDF2 method while retaining
A-stability. Let us first consider the \textit{error constant} of an LMM, which measures the leading term of its local truncation error; henceforth $|C|$ denotes the modulus of the error constant.
As Dahlquist showed, the A-stable trapezoidal method 
\begin{equation}\label{trapezoidalmethod}
    y_{n+1}-y_n=\frac{h}{2}\big(f(y_{n+1})+f(y_n)\big)
\end{equation}
has the smallest possible error constant for A-stable $2^\text{nd}$-order
LMMs with $|C|_\text{trap}=1/12$, being four times smaller than that of BDF2 with $|C|_\text{BDF2}=1/3$.
However, the trapezoidal method is not L-stable, and therefore is not effective for solving problems with highly stiff modes.  We refer to the rate at which an LMM damps infinitely-stiff modes as the \textit{asymptotic damping factor} (see Definition \ref{def:dampingfactoratinfinity}), denoted by $\Rinf$. 
The trapezoidal method has $\Rinf=1$ (no damping); an A-stable method with $\Rinf=0$
is said to be L-stable.  

Perhaps for historical reasons, in existing papers and texts, definitions of L-stability require A-stability as a pre-requisite (see e.g., \cite[Section IV.3]{wanner1996solving} or \cite[Section 8.3.2]{leveque2007finite}).  However, the condition $\Rinf=0$ is useful in many situations even if the method is not A-stable but only A($\alpha$)-stable \cite[Section V.2]{wanner1996solving} for some $\alpha\in[0,\pi/2)$.
More generally, such strong damping may not be necessary and it can be sufficient to have a damping factor significantly less than unity.
The notion of some finite amount of damping for stiff modes was introduced and studied by Jeltsch \cite{jeltsch1977multistep}.  Here, rather than the \emph{order} of damping at infinity we focus on the actual rate of damping, with the notion of L$(\kappa)$-stability to indicate that stiff modes are damped by a factor of $\kappa$ at each step
(i.e., $\Rinf\le \kappa$; see Section \ref{sec:stabdef}).


The main goal of the present work is to find  A-stable LMMs of order $p=2$ with $2\le k\le5$ steps that have (i) the error constant $|C|$ and (ii) the asymptotic damping factor $\Rinf$ as small as possible. As we will see, these are conflicting optimization criteria: one has to trade  $|C|$ for $\Rinf$---by slightly relaxing the L-stability constraint, one can make $|C|$ smaller, and vice versa.


In Section \ref{sec:notation}, we introduce some notation and recall certain recursive Schur--Cohn-type theorems to bound the zeros of polynomials  with respect to (closed or open) disks in the complex plane. 
In Section \ref{sec:error-constant}, we review the local consistency conditions and the definition of the error constant of LMMs.  
 In Section \ref{sec:stabdef}, we recall and define the stability 
 concepts encountered in later sections. 
Of the above stability concepts, the most difficult one to verify is A-stability. In Section \ref{sec:A-stability}, we present an algorithm to determine whether an LMM is A-stable.  Our result does not rely solely on the {root locus curve} of the LMM, which would provide insufficient information. 
Since the algorithm is not specific to the methods constructed here, it may also be useful for proving A-stability of other LMMs.
In Section \ref{sec:L-stability}, we describe algorithms to certify 
L$(\kappa)$-stability (for $\kappa\in[0,1]$) and strict L$(\kappa)$-stability (for $\kappa\in(0,1]$).
In Section \ref{sec:optimization}, we formulate the 
optimization problem that will be solved in Sections \ref{sec:bdflike}--\ref{sec:Lkappa}.



In Section \ref{sec:bdflike}, we focus on LMMs that are A-stable and L$(0)$-stable. The value of L$(0)$-stability lies in the fact that 
very stiff modes are damped very rapidly. By using symbolic techniques, we determine the unique methods with the smallest $|C|$ in this class for $3\le k\le 5$ steps, thus introducing the BDFL$k_2$ methods, see  Table \ref{tab:bdfl_345}; the BDFL$2_2$ method would be identical to the classical BDF2 method. (One can think of ``L'' in the acronym BDFL as ``BDF-like'' but also as ``L-stable''.) We remark that---due to the non-trivial nature of the optimization problem---for the BDFL$5_2$ method we could not prove optimality, but we have shown that its  error constant $|C|$ differs by at most $10^{-6}$ from the theoretical  optimum. Figure \ref{fig:err_to_step} highlights that the optimal error constant $|C|$ gets closer and closer to the theoretical minimum $|C|_\text{trap}=1/12$ as we choose L$(0)$-stable LMMs with larger step numbers $k$.

In Section \ref{sec:bdflike}, 
we studied methods with $\Rinf=0$.  In practice, good results can be
obtained even for many 
stiff problems 
 with LMMs that have a prescribed amount of damping $0<\Rinf<1$, 
since this also leads to the elimination of stiff oscillations
within a few steps. In Section \ref{sec:Lkappa}, we turn to the investigation of such methods. 
In Section \ref{sec:k=2strictlyLkappamethods}, we explicitly determine the unique optimal $2$-step $2^\text{nd}$-order A-stable and L$(\kappa)$-stable method for any given $\Rinf=\kappa\in(0,1)$. In Section \ref{sec:k=3someLkappamethods}, we study three restricted families of $3$-step $2^\text{nd}$-order A-stable and L$(\kappa)$-stable 
methods---the exploration of the full space of 3-step LMMs would be computationally demanding. Despite the fact that the explicit determination of the set of A-stable methods in these families is not always possible (only a larger set in the parameter space that contains all A-stable methods), or the minimization algorithms do not always yield answers in a reasonable amount of time, we can bypass these technical difficulties and still obtain the optimal error constants and the corresponding LMMs in the three families for any $\kappa\in(0,1)$. Some LMMs to be used for the numerical tests are presented in Table \ref{tab:lk_1234}.
The trade-off between $|C|$ and $\Rinf$ is compactly illustrated by
Figure \ref{fig:lk_errmin}, summarizing the theoretical findings of Section \ref{sec:Lkappa}. In Section~\ref{sec:spectral}, we relate our optimization criterion to
properties customarily used to assess time integrators in the structural dynamics literature (i.e., in the
numerical simulation of vibrating structures and mechanical systems): phase accuracy, low-frequency dissipation, and transient
overshoot. We show that the error constant determines the leading term of
the relative period error, so minimizing $|C|$ minimizes the low-frequency
phase error, while A-stability guarantees non-negative low-frequency
dissipation; several of our optimized methods have no dissipation at leading
order. Transient overshoot, in turn, depends on the root structure of the second
characteristic polynomial of the LMM: it is absent for the BDFL methods and the optimal 3-step methods,
but can occur for the optimal 2-step methods, especially for larger
$\kappa$.

Finally, in Section \ref{sec:tests}, we study the behavior
of some optimal schemes through numerical tests. Motivated by these results, in Section \ref{sec:vss} we develop a proof-of-concept variable-step-size implementation of some of our new methods, and test that implementation numerically. The development of efficient LMMs, along with effective step-size controllers, remains
an active area of research \cite{zhang2021improved,arevalo2021local,li2022stability,glandon2022linearly}. 
\paragraph{Relation to previous work.}
Several of the methods studied here have appeared before, each as the
minimizer of $|C|$ within a restricted class of methods. 
Our BDFL$3_2$
method coincides with the $2^\text{nd}$-order numerical differentiation formula
(NDF2) of Klopfenstein \cite{klopfenstein1971numerical}, the order-2
formula of MATLAB's \texttt{ode15s} \cite{shampine1997matlab}, which was
obtained as the optimal member of a one-parameter family.
Vatsa et
al.~\cite{vatsa2010re} combined BDF formulas of different orders and
obtained, among others, the methods BDF2OPT(4) and BDF2OPT(5), which coincide
with our BDFL$3_2$ and BDFL$4_2$ methods (their coefficients are not
displayed explicitly in \cite{vatsa2010re}). On the other hand, in structural dynamics, an
optimized 2-step family with a prescribed asymptotic spectral radius
$\rho_\infty$ goes back to \cite{masarati2001multistep,zhang2020astable}, and
\cite{zhang2021improved} extended this approach to
three and four steps (their LMS3 and LMS4 methods) by requiring all roots of
the second characteristic polynomial of the LMM to coincide at $-\rho_\infty$ and
minimizing $|C|$ within the resulting one- or two-parameter families. With
$\rho_\infty=\kappa$, their 2-step family coincides with our optimal 2-step
family, and LMS3 with $\rho_\infty=0$ coincides with BDFL$3_2$. 
In contrast with these works, we optimize over
\begin{itemize}
  \item all $2$-step $2^\text{nd}$-order A-stable and L$(\kappa)$-stable
    methods, for every $\kappa\in[0,1)$;
  \item all $k$-step $2^\text{nd}$-order A-stable and L-stable methods, for
    $k\in\{3,4,5\}$;
  \item three explicitly stated subclasses of $3$-step $2^\text{nd}$-order
    A-stable and L$(\kappa)$-stable methods, for $\kappa\in(0,1)$,
\end{itemize}
certifying A-stability and optimality symbolically. This proves that the
2-step family (whose optimality was already claimed in
\cite{zhang2020astable}) and the BDFL$3_2$ and BDFL$4_2$ methods are the
unique optima in their respective classes, and that the error constant of
BDFL$5_2$ is within $10^{-6}$ of the optimum. It also shows that LMS4 with
$\rho_\infty=0$ is not optimal, since BDFL$4_2$ has a smaller error constant. 
However, for $\kappa\in(0,1)$, the LMS3 family has a smaller error
constant than the 3-step methods we obtain in restricted classes; see 
Remarks~\ref{rem:ZZMLMS4}, \ref{rem:ZZM2step}, and~\ref{rem:ZZMLMS3},
where we also conjecture that the equal-root condition of
\cite{zhang2021improved} is necessary for optimality.\\

The main contributions of the present paper are the following:
\begin{itemize}
    \item development of a $2^\text{nd}$-order A-stable and L-stable five-step LMM with (essentially) optimal error constant---the BDFL$5_2$ method (see Table \ref{tab:bdfl_345});
    \item certification of the global optimality of the L-stable three- and
  four-step methods BDFL$3_2$ and BDFL$4_2$, previously obtained within
  restricted classes (BDFL$3_2$ in \cite{klopfenstein1971numerical,shampine1997matlab,vatsa2010re,zhang2021improved}, BDFL$4_2$ in \cite{vatsa2010re}); BDFL$4_2$ also improves on the
  L-stable LMS4 method of \cite{zhang2021improved} (see
  Table \ref{tab:bdfl_345});
   \item certification that a known 2-step family
\cite{masarati2001multistep,zhang2020astable,zhang2021improved} is the
  unique optimal $2$-step $2^\text{nd}$-order A-stable and L$(\kappa)$-stable
  method for every $\kappa\in[0,1)$, and determination of optimal 3-step $2^\text{nd}$-order A-stable and L$(\kappa)$-stable methods within three restricted classes (see
  Table \ref{tab:lk_1234});
        \item an analysis relating the error constant and the asymptotic
    damping factor to phase accuracy, low-frequency dissipation, and
    transient overshoot (see Section~\ref{sec:spectral} and
    Table~\ref{tab:dissipation_D});
    \item proof-of-concept variable-step-size implementations of the
  optimal BDFL$k_2$ methods (see Section \ref{subsec:robertson} and
  Appendix \ref{varstepbdfl45coeffs}).
\end{itemize}
Additionally, we provide a practical symbolic verification and optimization framework, implemented in the Wolfram Language \textit{Mathematica} version 14 and including
\begin{itemize}
    \item Algorithm \ref{Astabilitydeterminationalg} 
    to 
    rigorously check A-stability of a given LMM or a parametric family of LMMs;
    \item Algorithms \ref{vonNeumannalgorithm} and \ref{Schuralgorithm} to check (strict) L($\kappa$)-stability of LMMs---these algorithms test whether the roots of a univariate polynomial are included in the closed or open disks of radius $\kappa>0$ centered at the origin of the complex plane.
\end{itemize}

\subsection{Notation and preliminaries}\label{sec:notation}

We write $\mathbb{N}$ to denote the natural numbers \emph{including zero}, and $\imath$ to denote the complex imaginary unit.
The 
real part
and the complex conjugate are denoted by $\text{Re}(z)$
and $\bar{z}$ ($z\in\mathbb{C}$), and the closed left and right half-planes of $\mathbb{C}$ by 
$\{\text{Re}\le 0\}$ and $\{\text{Re}\ge 0\}$, respectively.
For $r>0$, $\mathbb{D}_r\subset\mathbb{C}$ denotes the open disk of radius centered at the origin. The boundary and closure of a set $H\subset\mathbb{C}$ are $\partial H$ and $\overline{H}$. The logical \textit{or} and logical \textit{and} symbols will sometimes be denoted by $\lor$ and $\land$ when displaying some formulae produced by  \textit{Mathematica}.

To simplify the terminology, when discussing the error constant $C\in\mathbb{R}$ of an LMM (to be defined in Section \ref{sec:error-constant}), by the \textit{smallest error constant} we actually mean the smallest value of the \textit{modulus} of the error constant.

The degree, leading coefficient, and constant coefficient of a univariate polynomial $Q$ are denoted by $\deg(Q)$, $\lc(Q)$, and $\cc(Q)$ respectively. As usual, the symbols $Q\equiv 0$ and $Q\not\equiv 0$ mean that $Q$ vanishes identically and its negation, respectively. 
 A \textit{dominant root} of $Q$  is a root with the largest modulus. The \textit{reduced polynomial} of  $Q(z)=\sum_{j=0}^{n}\gamma_j z^j$ ($z\in\mathbb{C}$, $\gamma_j\in\mathbb{C}$, $n\in\mathbb{N}^+$) is defined by
\begin{equation}\label{reducedpolydef}
Q^\textbf{r}(z):=\sum^n_{j=1}\left(\overline{\gamma_n}\gamma_j-\gamma_0\overline{\gamma_{n-j}}\right)z^{j-1}.
\end{equation}
Notice that $\deg(Q^\textbf{r})\le \deg(Q)-1$. 
Univariate polynomials whose roots lie in $\mathbb{D}_1$ are often called \textit{Schur} polynomials. A polynomial is \textit{von Neumann}, if its roots are located in $\overline{\mathbb{D}_1}$. 
A univariate polynomial $Q$ satisfies the \textit{root condition} \cite[Section III.3]{hairer1993solving} if its roots lie in $\overline{\mathbb{D}_1}$, and possible roots on $\partial\mathbb{D}_1$ have multiplicity $1$.
A recursive form of the classical Schur--Cohn test is reproduced, for example, in \cite{loczi2020optimal}.
\begin{thm}\label{thm:SCreduction}
    Let $Q$ be a univariate polynomial with $\deg(Q)\ge 2$ and $\lc(Q)\ne 0\ne \cc(Q)$.\\ Then
 $Q$ is a Schur polynomial if and only if
\[
 |\lc(Q)|>|\cc(Q)| \text{\ \ and\ \ } Q^\mathbf{r} \text{ is a Schur polynomial.}
\]
Moreover, $Q$ is a von Neumann polynomial if and only if
\[
\text{either } \Big( |\lc(Q)|>|\cc(Q)| \text{\ \ and\ \ } Q^\mathbf{r} \text{ is a von Neumann polynomial\,}\Big)
\]
\[
\text{or } \Big( Q^\mathbf{r}\equiv 0 \text{ and } Q' \text{ is a von Neumann polynomial\,}\Big).
\]
Furthermore, $Q$ satisfies the root condition if and only if 
\[
\text{either } \Big( |\lc(Q)|>|\cc(Q)| \text{\ \ and\ \ } Q^\mathbf{r} \text{ satisfies the root condition\,}\Big)
\]
\[
\text{or } \Big( Q^\mathbf{r}\equiv 0 \text{ and } Q' \text{ is a Schur polynomial\,}\Big).
\]
\end{thm}
\noindent The following elementary scaling property will be useful.
\begin{prop}\label{prop:kappadisk}
    Let $Q$ be a univariate polynomial with $\deg(Q)\ge 1$ and fix some $\kappa>0$. \\ Then
    \[
   \text{the roots of } Q \text{ are located in } \mathbb{D}_\kappa \ (\text{or in }\overline{\mathbb{D}_\kappa})
   \]
   \[
\Updownarrow 
\]
   \[\text{the polynomial } z\mapsto Q(\kappa z) \text{ is Schur } (\text{or von Neumann}).
\]
\end{prop}

\section{A-stable LMMs with small error constant and damping factor} \label{sec:setup}
Since any ODE system can be written as an equivalent autonomous ODE system, we consider the initial-value ODE
\begin{align}
    y'(t) & = f(y(t)) &  y(0) & = y_0, \label{ODE}
\end{align}
where $f: \real^m \to \real^m$ is assumed to be Lipschitz continuous and smooth on some appropriate domain.
We will solve \eqref{ODE} numerically, obtaining
a sequence of numerical approximations $y_n \approx y(nh)$ where $h>0$ is the step size.
We write an LMM of $k\in\mathbb{N}^+$ steps as 
\begin{align} \label{LMM}
    \sum_{j=0}^k \alpha_j y_{n+j} & = h \sum_{j=0}^k \beta_j f(y_{n+j})\quad\quad\quad(n\in\mathbb{N})
\end{align}
with $\alpha_k\ne 0$ (guaranteeing  solvability) and $|\alpha_0|+|\beta_0|\ne 0$ (ensuring that the order of the recursion is not less than $k$). The coefficient lists 
$(\alpha_0,\alpha_1,\ldots,\alpha_k)$ and $(\beta_0,\beta_1,\ldots,\beta_k)$
will often be referred to as the $\alpha$- and $\beta$-\textit{coefficients} of the LMM (note the increasing ordering). From now on, we
use the normalization \begin{equation}\label{normalizationassumption}
    \alpha_k = 1.
\end{equation}

An LMM is \textit{explicit} if $\beta_k=0$, or \textit{implicit} if $\beta_k\ne 0$; these classes will be abbreviated as ELMMs and ILMMs, respectively.

Recall the \textit{stability polynomial}  $\pi(\,\cdot\,; z)$ of the LMM, which depends on the parameter $z\in\mathbb{C}$ and is given by \cite[Section 7.3]{leveque2007finite}
\begin{align} \label{stabfun}
    \pi(\zeta; z) & := \rho(\zeta) - z \sigma(\zeta) \quad\quad\quad (\zeta\in\mathbb{C}),
\end{align}
where
\begin{equation}
    \label{charpolys}
    \rho(\zeta)  := \sum_{j=0}^k \alpha_j \zeta^j \quad\quad\quad\text{ and }\quad\quad\quad  \sigma(\zeta)  := \sum_{j=0}^k \beta_j \zeta^j
\end{equation} 
are the \textit{first} and \textit{second characteristic polynomials} of \eqref{LMM}, respectively.
\begin{dfn}[{\cite[Section III.2]{wanner1996solving}}]
    An LMM \eqref{LMM} is irreducible if the polynomials $\rho$ and $\sigma$ have no common root; otherwise, it is reducible.
\end{dfn}
\noindent Since a reducible $k$-step LMM is essentially equivalent to a $\widetilde{k}$-step LMM with $\widetilde{k}<k$, from now on, we assume
\begin{equation}\label{irredasusumption}
    \rho \text{ and } \sigma \text{ have no common root.}
\end{equation}

\subsection{Local order of consistency and the error constant} \label{sec:error-constant}
The local truncation error \cite[Section 5.9.1]{leveque2007finite} for the LMM \eqref{LMM} is   
\begin{align} \label{localtruncationerror}
    \text{LTE} & = \frac{1}{h} \left(\sum_{j=0}^k \alpha_j\right) y(t_n) + \sum_{p=1}^\infty h^{p-1} \sum_{j=0}^k \left( \frac{j^p}{p!}\alpha_j - \frac{j^{p-1}}{(p-1)!}\beta_j \right)y^{(p)}(t_n).
\end{align}
For a method of $2^\text{nd}$-order accuracy ($\text{LTE}=\order(h^2)$) we therefore require
\begin{subequations} \label{consistency}
\begin{align}
    \sum_{j=0}^k \alpha_j & = 0 \label{consistencyA} \\
    \sum_{j=0}^k \left(j \alpha_j - \beta_j\right) & = 0 \\
    \sum_{j=0}^k \left(\frac{1}{2}j^2 \alpha_j - j \beta_j\right) & = 0.
\end{align}
\end{subequations}
We say a method is \emph{consistent} if it satisfies the first two conditions of \eqref{consistency}.
One might be inclined to define the \emph{error constant} of the method as simply the coefficient of the first non-vanishing term in \eqref{localtruncationerror}.
A more appropriate estimate of the error can be obtained through a judicious normalization of this residual, see the discussion in \cite[Section III.2]{hairer1993solving}.

\begin{dfn}\label{def:errorconstant}
The error constant $C\in\mathbb{R}$ of the LMM \eqref{LMM} of order $p\in\mathbb{N}^+$ is
\[    C:=\frac{C_{p+1}}{\sigma\left(1\right)},
\]
where   
\[
C_{p+1}:=\sum^k_{j=0}\left( \frac{j^{p+1}}{(p+1)!}
 \alpha_j- \frac{j^p}{p!}\beta_j\right)
\]
and $\sigma$ is the second characteristic polynomial  \eqref{charpolys} of the LMM.
\end{dfn}
\begin{rem}\label{rem:sigma1notzero}
    It can be shown that consistency and zero-stability together imply $\sigma(1)\ne 0$,
    so the error constant is always well-defined.
\end{rem}


\subsection{Some stability concepts} \label{sec:stabdef}

 Some classical stability concepts \cite[Section III.3]{hairer1993solving},  \cite[Section V.1]{wanner1996solving} are reviewed next.

\begin{dfn}\label{def:0-stability}
    An LMM \eqref{LMM} is zero-stable if the first characteristic  polynomial $\rho$ satisfies the root condition.
\end{dfn}
\begin{dfn}\label{def:stabilityregion}
    The region of absolute stability $\stab$ of an LMM \eqref{LMM} is 
    \[
    \stab:=\{z\in\mathbb{C}\ :\ \deg(\pi(\,\cdot\,;z))=k\text{\ \ and\ \  } \pi(\,\cdot\,;z)\text{ satisfies the root condition}\},
    \]
    where $\pi$ is the stability polynomial \eqref{stabfun}.
\end{dfn}
\begin{dfn}\label{def:A-stability}
    An LMM \eqref{LMM} is A-stable if\ \,$\{\mathrm{Re}\le0\}\subseteq\stab$.
\end{dfn}
\begin{rem} Zero-stability is equivalent to $0\in\stab$, and A-stability implies zero-stability. 
\end{rem}
 In analogy with the stability function of a Runge--Kutta method---describing the behavior of the method on a scalar linear Dahlquist test problem and often denoted by $R(z)$ for $z\in\mathbb{C}$, see \cite[Section IV.3]{wanner1996solving}---we define the \textit{damping factor} of an LMM.
\begin{dfn}\label{def:dampingfactor}
     Given an LMM \eqref{LMM} and some $z\in\mathbb{C}$, the damping factor $R(z)$ is the modulus of the dominant root of the polynomial $\pi(\cdot; z)$. 
\end{dfn}

\begin{prop}\label{LMMdampingfactorprop} For a given LMM \eqref{LMM}, we have the following.
\begin{itemize}
    \item If the method is explicit and consistent, then its damping factor $R(z)$ becomes unbounded as $z\to\infty$ (here, and in what follows, $\infty$ denotes complex infinity). 
    \item If the method is implicit, then the limit $\lim_{z\to\infty}R(z)$ exists and is equal to the modulus of the dominant root of the second characteristic polynomial $\sigma$. 
\end{itemize} 
\end{prop}
\begin{proof}
  The stability polynomial \eqref{stabfun} of an ELMM takes the form
  $  \pi(\zeta;z)=\alpha_k \zeta^k+ \sum_{j=0}^{k-1} (\alpha_j -z  \beta_j )\zeta^j$, 
  and there is an index $0\le j_0\le k-1$ for which $\beta_{j_0}\ne 0$ (otherwise, $\sigma(1)=0$ would contradict the convergence of the LMM, see Remark \ref{rem:sigma1notzero}). This means that the coefficient of $\zeta^{j_0}$ becomes unbounded as $z\to\infty$. So, due to Vieta's formulae, a certain elementary symmetric polynomial of the roots of $\pi(\cdot;z)$ becomes unbounded as $z\to\infty$. But this can only happen if at least one root becomes unbounded, hence the modulus of the dominant root of $\pi(\cdot;z)$, $R(z)$ becomes unbounded.
  
  Let us consider an ILMM. For any $z\ne 0$, set $\widetilde{z}:=1/z$. Then the roots of equations
   $\pi(\cdot;z)=0$ and $-\widetilde{z} \,\pi(\cdot;z)=0$ coincide. But 
   $-\widetilde{z} \,\pi(\cdot;z)= \sigma(\cdot)-\widetilde{z}\rho(\cdot)
   $, and the roots of a polynomial are continuous functions of its coefficients, so the modulus of the dominant root of  $\pi(\cdot;z)$ converges to the modulus of the dominant root of $\sigma$ as $\widetilde{z}\to 0$.
\end{proof}
\noindent Since we are not interested in LMMs whose damping factor can become unbounded, from now on we assume
\begin{equation}\label{implicitasusumption}
\beta_k \ne 0.
\end{equation}

 Based on Proposition \ref{LMMdampingfactorprop}, we define the \textit{asymptotic damping factor} as follows.
\begin{dfn}\label{def:dampingfactoratinfinity}
     Given an ILMM \eqref{LMM}, its damping factor at $\infty$ is
     \[
     \Rinf:=\lim_{z\to\infty}R(z) = \text{ the modulus of the dominant root of\ \,}  \sigma.
     \]
\end{dfn}
 \noindent Next, we recall the concept of L-stability.
\begin{dfn}
    An ILMM \eqref{LMM} is L-stable, if it is A-stable and $\Rinf=0$. 
\end{dfn}
\noindent We generalize the above concept as follows---notice that A-stability is not assumed below.

\begin{dfn}\label{Lkappadef}
    Let $\kappa\in [0,+\infty)$ be given. An 
    ILMM \eqref{LMM}  is said to be L$(\kappa)$-stable, if the asymptotic damping factor satisfies $\Rinf\le \kappa$.
\end{dfn}
\begin{rem}
    Note that this is quite different from the concept of L$(\theta)$-stability introduced in
    \cite{Pagano2026}.
\end{rem}
\begin{dfn}\label{strictLkappadef}
    Let $\kappa\in (0,+\infty)$ be given. An 
    ILMM \eqref{LMM}  is said to be strictly L$(\kappa)$-stable, if the asymptotic damping factor satisfies $\Rinf< \kappa$.
\end{dfn}
\noindent It is clear  that for an ILMM \eqref{LMM} we have the following:
\begin{equation}\label{Lkappacharacterization}
\text{ L}(\kappa)\text{-stability }\Longleftrightarrow \text{ the roots of } \sigma \text{ are in }\overline{\mathbb{D}_\kappa};
\end{equation}
\[
 \text{ strict L}(\kappa)\text{-stability }\Longleftrightarrow \text{ the roots of } \sigma \text{ are in }{\mathbb{D}_\kappa};
\]
\[
 \text{ L-stability } \Longleftrightarrow \text{ A-stability and L}(0)\text{-stability}.
\]

\begin{rem}
    In addition to the above, an LMM is defined (see \cite[Sections V.1 and V.7]{wanner1996solving}) to be stable at $\infty$ if $\sigma$ satisfies the root condition, and 
    strictly stable at $\infty$ if $\sigma$ is a Schur polynomial. In our terminology,
    \[
    \text{ strict stability at }\infty \Longleftrightarrow \text{ strict L}(1)\text{-stability}.
    \]    
\end{rem}


%
%
%

\subsection{Characterization of A-stable methods}
\label{sec:A-stability}

Checking A-stability of a given LMM is a non-trivial task, since it is a non-local property. To investigate A-stability or to visualize the region of absolute stability $\stab$ of an LMM, the following concept is often introduced \cite[Section 7.6.1]{leveque2007finite}.
\begin{dfn} \label{def:rootlocuscurve}
    The root-locus curve (RLC, or boundary-locus curve) of an LMM \eqref{LMM} is defined as the image of $\partial\mathbb{D}_1$ under the map ${\rho}/{\sigma}$. A parametrization of the RLC is given by
    \[
    [0,2\pi]\ni\theta\mapsto \frac{\rho(e^{\imath\theta})}{\sigma(e^{\imath\theta})}\in\mathbb{C}.
    \]
\end{dfn}
\noindent However, it is important to keep in mind that, in general, the boundary of the region of absolute stability, $\partial \stab$, is \textit{not} the same as the RLC. It should also be realized that the following ``implication''
\[
\text{``RLC} \subseteq \{\mathrm{Re}\ge0\}\ \  \Longrightarrow \text{\ \  the LMM is A-stable''}
\]
is, in general, \textit{false}---despite the fact that some papers in the literature seem to use this necessary condition (see our Remark \ref{rem:Astabneccond} below) to conclude that the LMM is A-stable.

\begin{ex} Consider, for example, the $2^\text{nd}$-order LMM
    \[y_{n+3}+4y_{n+2}-7y_{n+1}+2y_n=h \big(f(y_{n+3})+3 f(y_{n+2})\big)\]
for which the RLC is a subset of the closed right half-plane, but this LMM is not A-stable (it is not even zero-stable).
\end{ex}

In order to conclude A-stability, one needs an additional condition such as strong zero-stability, or the existence of a stable point in the open left half-plane.
We will instead make use of the following characterization.
\begin{thm}[{\cite[Section V.1, Theorem 1.3]{wanner1996solving}}]\label{Astabilitycharacterization}
    If the LMM \eqref{LMM} is A-stable, then
    \begin{align} \label{astablecond}
    \mathrm{Re}\left(\frac{\rho\left(\zeta\right)}{\sigma\left(\zeta\right)}\right)>0 \quad\text{\ \ for any\ \  }\zeta\in\mathbb{C}\text{\ \ with\ \ }|\zeta|>1\text{.}
    \end{align}
    For irreducible methods the converse is also true: \eqref{astablecond} implies A-stability.
\end{thm}
\begin{rem}\label{rem:Astabneccond}
    By switching to polar coordinates, we have $\zeta=r e^{i\theta}$ with  $r>1$ and $\theta\in[0,2\pi]$ in \eqref{astablecond}, so taking the limit as $r\to 1^+$, we get that 
    \[
    \text{A-stability\ \ } \Longrightarrow \emph{\ \ RLC} \subseteq \{\mathrm{Re}\ge0\}.
    \]
\end{rem}
\medskip
Theorem \ref{Astabilitycharacterization} can be implemented in the Wolfram Language (\textit{Mathematica} version 14) as follows. 
\begin{algo}\label{Astabilitydeterminationalg}\ 
\begin{tcolorbox}[colframe=black, colback=white, sharp corners]
\emph{\textbf{Step 1.}}
    Given an irreducible ILMM \eqref{LMM} with characteristic polynomials $\rho$ and $\sigma$, define
\begin{flalign*}
 & \emph{\texttt{realPartInPolar\,=\,Simplify\,@\,ComplexExpand\,@\,Re[\,$\rho$[$\zeta$]/$\sigma$[$\zeta$] /. $\zeta\to$ r\,Exp[I\,$\theta$]\,]}}& 
 \end{flalign*}
\emph{\textbf{Step 2.}} Convert the trigonometric expressions to polynomials  using the Weierstrass substitution $\theta=2\arctan(t)$:
\begin{flalign*}
 & \emph{\texttt{rationalFunction\,=\,Simplify\,@\,TrigExpand[\,realPartInPolar
 \,/. $\theta\to$ 2 ArcTan[t]]\,}}  & 
 \end{flalign*}
\emph{\textbf{Step~3.}} Eliminate the two quantifiers $\forall r>1$ and $\forall t\in\mathbb{R}$ from the inequality \emph{\texttt{rationalFunction > 0}}:
\begin{flalign*}
 &  \emph{\texttt{Reduce[\,ForAll[\{t, r\}, r > 1, rationalFunction > 0], Reals\,]}}
 & 
 \end{flalign*}
\emph{\textbf{Output.}} If the output of \emph{\texttt{Reduce}} is \emph{\texttt{True}}, then the LMM is A-stable. If the output is \emph{\texttt{False}}, then A-stability does not hold. When the LMM coefficients contain some parameters, the output of \texttt{\emph{Reduce}} is a set of necessary and sufficient conditions on the $\alpha$- and $\beta$-coefficients such that the corresponding LMM is A-stable.
\end{tcolorbox}
\end{algo}

\begin{rem}\label{rem:Astabilitywithparameters}
 The code checks A-stability of a given LMM within seconds on a standard laptop. In later sections, we use this approach in families of LMMs, with the parameters being certain $\alpha$- or $\beta$-coefficients.
\end{rem}
\begin{rem}
   In certain special cases, instead of the general Weierstrass substitution in Step 2 above, we can use other substitutions (see Section \ref{sec:bdflikek4k5}) to speed up the algorithm. 
\end{rem}



\subsection{Characterizations of L($\kappa$)-stability}
\label{sec:L-stability}

A special class of 
ILMMs is defined below. 
\begin{dfn}\label{BDFLdefinition}
    A $k$-step 
    ILMM \eqref{LMM} is said to be BDF-like (BDFL for short) if
    \begin{align} \label{bdflike}
    \beta_0 = \beta_1 = \cdots = \beta_{k-1}=0, \quad \beta_k\ne 0.
    \end{align}
    When the step number $k\ge 2$ and order $p\ge 1$ are also indicated, we use the notation BDFL$k_p$. 
\end{dfn}
\begin{rem}
    For the classical BDF methods, we have $k=p$, but herein we focus on methods with $p=2$ and $k>2$.
\end{rem}
\noindent Due to \eqref{Lkappacharacterization}, the following equivalence holds:
\begin{equation}\label{L0characterization}
\text{a given ILMM } \eqref{LMM} \text{ is L}(0)\text{-stable }\Longleftrightarrow \text{ the ILMM is BDF-like}.
\end{equation}

\noindent In particular, when designing LMMs, imposing L$(0)$-stability is straightforward.\\ 

However, devising conditions on the coefficients to ensure L($\kappa$)-stability when $\kappa>0$ is more challenging: for a given ILMM with second characteristic polynomial $\sigma$, we utilize Proposition \ref{prop:kappadisk} and Theorem \ref{thm:SCreduction} to get the following.
\begin{algo}\label{vonNeumannalgorithm} \ 
\begin{tcolorbox}[colframe=black, colback=white, sharp corners]
\emph{\textbf{Step 1.}} Define $Q(\zeta):=\sigma(\kappa \zeta)$.\\
\emph{\textbf{Step 2.}} Generate its reduced polynomial $Q^\mathbf{r}$ \eqref{reducedpolydef}.\\
\emph{\textbf{Step 3.}} 
\begin{itemize}
    \item \emph{\textbf{Case 1:}} If $Q^\mathbf{r}\not\equiv 0$, then check whether $|\lc(Q)|>|\cc(Q)|$ and $Q^\mathbf{r}$ is a von Neumann polynomial. If yes, then the ILMM is L($\kappa$)-stable. If no, then the ILMM is not L($\kappa$)-stable.
     \item \emph{\textbf{Case 2:}} If $Q^\mathbf{r}\equiv 0$, then check whether $Q'$ is a von Neumann polynomial. If yes, then the ILMM is L($\kappa$)-stable. If no, then the ILMM is not L($\kappa$)-stable.
\end{itemize}
\emph{\textbf{Recursion.}} In both Cases 1--2, the degree of $Q$ is reduced by at least $1$ (the $\cdot^\mathbf{r}$ operator can reduce the degree by more than $1$), so we recursively continue the decision process from Step 2 with $Q_\text{new}:=Q^\mathbf{r}$ (Case 1) or $Q_\text{new}:=Q'$ (Case 2) until we reach degree 1, when the linear polynomial can be solved to see if its root is in $\overline{\mathbb{D}_1}$ or not. 
\end{tcolorbox}
\end{algo}

\begin{rem}
    When applying the above process to ILMMs where the $\beta$-coefficients contain some parameters, the coefficients of successive reduced polynomials will themselves become more and more complicated polynomials in the parameters.    
\end{rem}

 Checking strict L($\kappa$)-stability for $\kappa>0$ is technically  simpler, although it still involves a recursive step. 
\begin{algo}\label{Schuralgorithm} \ 
\begin{tcolorbox}[colframe=black, colback=white, sharp corners]
\emph{\textbf{Step 1.}} Define $Q(\zeta):=\sigma(\kappa \zeta)$.\\
\emph{\textbf{Step 2.}} Generate its reduced polynomial $Q^\mathbf{r}$ \eqref{reducedpolydef}.\\
\emph{\textbf{Step 3.}} Check whether $|\lc(Q)|>|\cc(Q)|$ and $Q^\mathbf{r}$ is a Schur polynomial. If yes, then the ILMM is strictly L($\kappa$)-stable. If no, then the ILMM is not strictly L($\kappa$)-stable.
\end{tcolorbox}
\end{algo}

\subsection{Optimization}  \label{sec:optimization}
Taking into account all of the foregoing, and setting $p=2$ in Definition \ref{def:errorconstant} in particular, we arrive at the following optimization problem.

\begin{problem}
Given $\kappa\in[0,1]$ and $k\ge 2$, find 
$\alpha_j$ and $\beta_j$ (for $0 \le j \le k)$ to
minimize 
\begin{equation}\label{errorconstantmodulustobeminimized}
|C|:=\left| \frac{1}{\sigma(1)}\sum_{j=0}^k \left(\frac{j^3}{3!} \alpha_j - \frac{j^2}{2!} \beta_j\right)\right|
\end{equation}
subject to 
\begin{itemize}
    \item[\circled{0}] normalization, irreducibility and implicitness: \eqref{normalizationassumption}, \eqref{irredasusumption}, \eqref{implicitasusumption};
    \item[\circled{1}] $2^{\text{nd}}$-order consistency: \eqref{consistency};
    \item[\circled{2}] A-stability: Definition \ref{def:A-stability};
    \item[\circled{3}] L$(0)$-stability  or (strict) L$(\kappa)$-stability (for $\kappa>0$): Definitions \ref{Lkappadef}--\ref{strictLkappadef};
    \item[\circled{4}] structural conditions: $\beta_j=0$ for $j\in S_\text{ind}$ with some index set $S_\text{ind}\subseteq\{0,1,\ldots,k-1\}$.
\end{itemize}
\end{problem}
\begin{rem}\label{rem:structuralconditions}
The structural conditions are introduced in order to make the problem more tractable.  This optimization problem involves a complicated set in the $2k$-dimensional parameter space of $\alpha$- and $\beta$-coefficients.  For larger values of $k$, its solution (particularly the characterization of the A-stable set) may become intractable.  We therefore resort, where necessary, to studying lower-dimensional subsets by setting some values $\beta_j$ to zero \emph{a priori}.  This enables us to find useful methods, even if we cannot guarantee that they are optimal over the larger set.
\end{rem}

\section{Optimal $2^\text{nd}$-order A-stable and L$(0)$-stable BDFL methods with $k\in\{3,4,5\}$ steps} \label{sec:bdflike}
In this section, we 
solve the optimization problem posed in Section \ref{sec:optimization} for $\kappa=0$ in the families of $k$-step ILMMs with $k\in\{3,4,5\}$. 
According to \eqref{L0characterization}, the L$(0)$-stability assumption implies that the ILMMs are BDFL methods. 

\begin{rem}
    For $k=2$, we recover the classical BDF2 method \eqref{BDF2LMMformula}; see Corollary \ref{cor:k=2Lkappaoptimalmethods}.
\end{rem}




\subsection{Three-step methods}\label{sec:bdflikek3}

In light of the BDF-like condition \eqref{bdflike} and the normalization condition \eqref{normalizationassumption}, we have four free parameters: $\alpha_0, \alpha_1, \alpha_2$, and $\beta_3$.  The three consistency conditions
\eqref{consistency} allow $(\alpha_0,\alpha_1,\alpha_2)$ to be rewritten in terms of $\beta_3$.  With this parameterization, we find that the method is irreducible as long as $\beta_3 \ne 2/3$.  Next, we apply Algorithm \ref{Astabilitydeterminationalg}, which reveals that the method is A-stable for $\beta_3 \in [3/5, 1]$.  Using \textit{Mathematica}'s \texttt{Minimize} command, we find that $\beta_3=3/5$ yields the minimum error constant \eqref{errorconstantmodulustobeminimized}, with value $|C|_\text{opt}=1/6$.

Thus, this method achieves an error constant equal to half of that of BDF2, at the cost of using an additional step.
It coincides with the $2^\text{nd}$-order numerical differentiation formula
(NDF2) of Klopfenstein \cite{klopfenstein1971numerical}, which is used at
order 2 in MATLAB's \texttt{ode15s} \cite{shampine1997matlab}; it was also
proposed in \cite{vatsa2010re}, and appears as the special case
$\rho_\infty=0$ of the LMS3 family of \cite{zhang2021improved}. It is
displayed in Table \ref{tab:bdfl_345}.

\subsection{Four- and five-step methods}\label{sec:bdflikek4k5}

After applying the normalization assumption and the consistency condition, we are left with two parameters $(\alpha_3,\beta_4)$ for $k=4$, and three parameters $(\alpha_3,\alpha_4,\beta_5)$ for $k=5$. To reduce the computing time, we have modified our process as follows.
\begin{enumerate}
    \item Instead of the general Weierstrass substitution in Algorithm \ref{Astabilitydeterminationalg}, 
    we convert the trigonometric expressions to polynomials in a simpler way via the substitutions 
    \begin{equation}\label{simplersubstitution}
        c:=\cos(\theta) \text{\ \  and \ \ } \sin^{2\ell}(\theta)=(1-c^2)^\ell\quad (\ell\in\mathbb{N}^+), 
    \end{equation}
     since only even powers of $\sin(\theta)$ are present.
    \item Instead of the characterization of A-stability (Theorem \ref{Astabilitycharacterization}), we use only the necessary condition given in Remark \ref{rem:Astabneccond}, that is, we apply \texttt{Reduce} with $r=1$ and a single quantifier 
    $\forall c\in[-1,1]$ to simplify the inequality \texttt{rationalFunction $\ge$ 0}. Let $\mathcal{P}$ denote the parameter set returned by \texttt{Reduce}, and $\mathcal{A}$ denote the parameter set of A-stable methods; then we have $\mathcal{A}\subseteq\mathcal{P}$.  The explicit description of $\mathcal{A}$ is unknown.
    \item We use \texttt{Minimize} to find the optimal value of the objective function \eqref{errorconstantmodulustobeminimized} over $\mathcal{P}$. The minimum turns out to be unique. Clearly, $\displaystyle \inf_\mathcal{A}|C|\ge\min_\mathcal{P}|C|$.
    \item Finally, we verify that the candidate ILMM (with the smallest $|C|$ in $\mathcal{P}$) is in fact irreducible and A-stable by using Algorithm \ref{Astabilitydeterminationalg} itself for this specific method (with concrete $\alpha$- and $\beta$-coefficients, that is, without parameters). Hence, $\displaystyle \inf_\mathcal{A}|C|=\min_\mathcal{P}|C|$.
\end{enumerate}

\subsection{Summary}

\noindent The optimal and essentially optimal BDFL methods are summarized in Table \ref{tab:bdfl_345}. Figure \ref{fig:err_to_step} shows a comparison of the error constant and steps, while Figures \ref{fig:bdfl3}--\ref{fig:bdfl5} display their regions of absolute stability (Definition \ref{def:stabilityregion}).

\begin{table}
\centering
\renewcommand{\arraystretch}{1.3}

\begin{tabular}{c|p{0.53\textwidth}|c}
\textbf{Steps} & \textbf{LMM} & \textbf{Error constant $C$} \\
\hline
3 &
$y_{n+3}=\frac{3}{2}y_{n+2}-\frac{3}{5}y_{n+1}+\frac{1}{10}y_n+\frac{3}{5}h f(y_{n+3})$ &
$-\frac{1}{6}$ \\
\hline
4 &
$y_{n+4}=\frac{2}{41}(18+11\sqrt{2})y_{n+3}
-\frac{2}{41}(1+12\sqrt{2})y_{n+2}
+\frac{2}{41}(10-3\sqrt{2})y_{n+1}
-\frac{1}{41}(13-8\sqrt{2})y_n
+\frac{4}{41}(10-3\sqrt{2})h f(y_{n+4})$ &
$\frac{1}{\sqrt{2}}-\frac{5}{6}\approx -0.12623$ \\
\hline
5 &
$y_{n+5}=\frac{1}{44}(45+14\sqrt{5})y_{n+4}
-\frac{\sqrt{5}}{2}y_{n+3}
+\frac{1}{88}(25+9\sqrt{5})y_{n+2}
-\frac{1}{22}(10-3\sqrt{5})y_{n+1}
+\frac{1}{88}(13-5\sqrt{5})y_n
+\frac{5}{44}(7-\sqrt{5})hf(y_{n+5})$ &
$\frac{1}{2\sqrt{5}}-\frac{1}{3}\approx -0.10973$
\end{tabular}

\caption{The optimal BDFL$k_2$ methods for $k\in\{3,4\}$ and the essentially optimal BDFL$5_2$ method (Remark \ref{rem:k=5complexity}) together with their associated error constants, see Section \ref{sec:bdflike}
}
\label{tab:bdfl_345}
\end{table}


\begin{figure}
    \centering
    \includegraphics[width=0.75\linewidth]{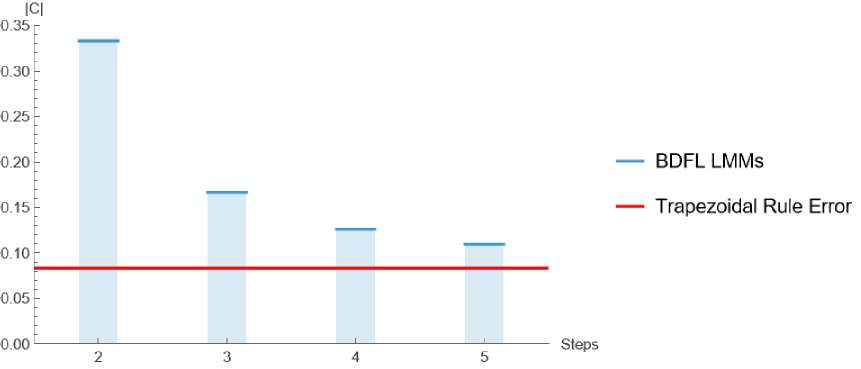}
    \caption{The modulus of the optimal error constant \eqref{errorconstantmodulustobeminimized} 
    of the BDFL methods in Table \ref{tab:bdfl_345}. The classical BDF2 method has $|C|=1/3$, and the trapezoidal method yields the theoretically best error constant.}
    \label{fig:err_to_step}
\end{figure}

\begin{rem}\label{rem:ZZMLMS4}
Optimized L-stable methods with three and four steps also appear in
\cite{zhang2021improved}, as the special case $\rho_\infty=0$ of their LMS3
and LMS4 families. Their L-stable LMS3 method is exactly our BDFL$3_2$.
Their L-stable LMS4 method belongs to the same class as BDFL$4_2$, namely
$\sigma(\zeta)=\beta_4\zeta^4$, but its error constant $|C|=2/15\approx
0.1333$ is larger than that of BDFL$4_2$, $|C|=5/6-1/\sqrt{2}\approx
0.1262$. The improvement is due to our use of rigorous global optimization.
\end{rem}
\begin{rem}\label{rem:k=5complexity}
    The complexity of obtaining the methods in Table  \ref{tab:bdfl_345} increases significantly with $k$. As we have seen in Section \ref{sec:bdflikek3}, the problem of optimality for $k=3$ is almost trivial. For $k=5$, however, the set $\mathcal{A}$ of A-stable methods 
    in the parameter space $(\alpha_3,\alpha_4,\beta_5)$ has been described by a logical combination of inequalities containing more than 200 disjoint components with boundaries being complicated  algebraic surfaces, and the description itself consisting of more than 2 MB of characters. \texttt{\emph{Minimize}} was then applied to these components, and the method for $k=5$ in Table \ref{tab:bdfl_345} was found. For certain components, however, \texttt{\emph{Minimize}} did not return a result in a reasonable amount of  time, so those components were discarded. Hence, we cannot claim that the BDFL$5_2$ method in Table \ref{tab:bdfl_345} is optimal; it is, however,  essentially optimal in the following sense.
    
    It is not difficult to show directly with \emph{\texttt{Reduce}} that the set $\mathcal{A}$  is a subset of the box 
    \[\left[-\frac{2}{35} \left(1+\sqrt{561}\right),\frac{2}{35} \left(-1+\sqrt{561}\right)\right]\times\left[-\frac{1}{5} \left(5+\sqrt{15}\right),\frac{1}{3} \left(\sqrt{7}-1\right)\right]\times\left[\frac{2}{35} \left(14-\sqrt{21}\right), 2\right]
    \] 
    in the parameter space $(\alpha_3,\alpha_4,\beta_5)$; in particular, $\beta_5>0$. By ignoring a non-negative factor $\frac{2(c-1)^2}{\beta_5}$, \texttt{\emph{rationalFunction}} in Step 2 of Section \ref{sec:bdflikek4k5} for $k=5$ becomes the polynomial 
    \begin{equation}\label{rem33poly}
    c^3 \left(-8 \alpha _3-24 \alpha _4+28 \beta _5-48\right)+ c^2 \left(-4 \alpha _3-16 \alpha _4+24
   \beta _5-36\right)+
   \end{equation}
   \[c \left(4 \alpha _3+10 \alpha _4-6 \beta _5+16\right)+(\alpha _3+4 \alpha _4-4 \beta _5+8).\]
   Now, instead of requiring the non-negativity of \eqref{rem33poly} for any $c\in[-1,1]$, we just require non-negativity for finitely many values of $c$, say, by using an equidistant grid of spacing $1/500$. 
   In this way, we obtain $1001$ linear inequalities in the parameters $(\alpha_3,\alpha_4,\beta_5)$. Let $\mathcal{P}$ denote the convex and bounded polyhedron defined by the intersection of these 1001 closed half-spaces. Then, again, we have $\mathcal{A}\subseteq\mathcal{P}$. Now, by using only rational arithmetic, \texttt{\emph{Minimize}} can quickly (in less than 30 seconds) determine that $\displaystyle \min_\mathcal{P}|C|=\frac{53861}{490866}=:\mu_{\mathcal{P},5}$. But the error constant $|C|=|\frac{1}{2\sqrt{5}}-\frac{1}{3}|$ of the irreducible and A-stable BDFL$5_2$ method in Table \ref{tab:bdfl_345} satisfies
   \[
   \left|\frac{1}{2\sqrt{5}}-\frac{1}{3}\right|\in \left(\mu_{\mathcal{P},5}, \ \mu_{\mathcal{P},5}+ 10^{-6}\right),
   \]
   so the chain of inequalities $ \left|\frac{1}{2\sqrt{5}}-\frac{1}{3}\right|\ge \displaystyle \inf_\mathcal{A}|C|\ge\min_\mathcal{P}|C|=\mu_{\mathcal{P},5}$
    shows that the error constant $|C|$ of the BDFL$5_2$ method differs by at most $10^{-6}$ from the theoretical global optimum $\inf_\mathcal{A}|C|$.
\end{rem}

\begin{figure}
     \centering
     \begin{subfigure}[b]{0.49\textwidth}
         \centering
         \includegraphics[width=\linewidth]{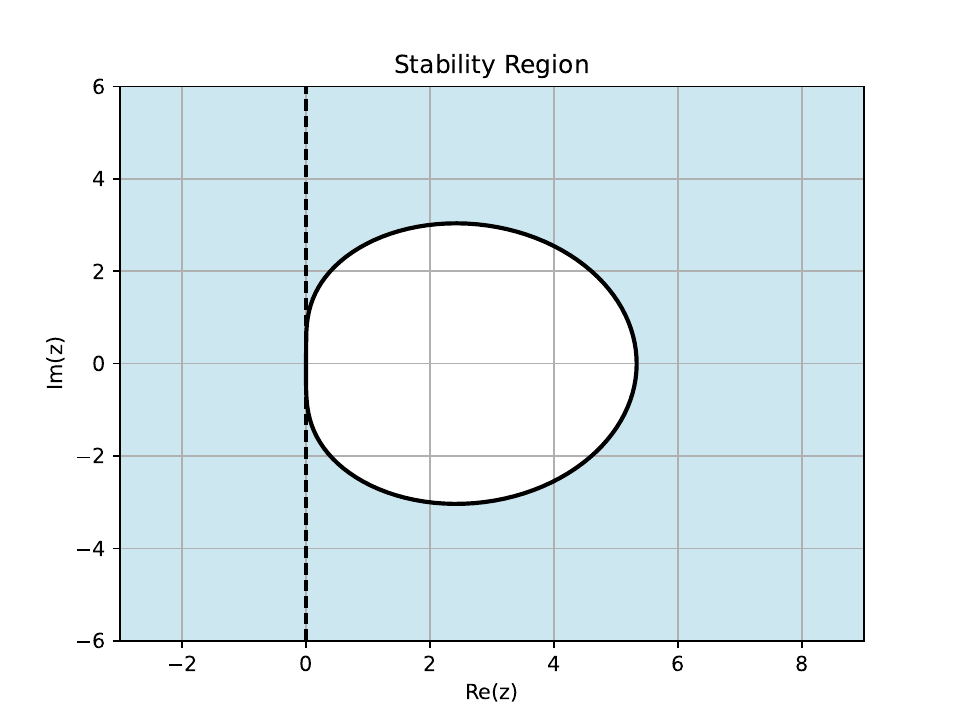}
     \end{subfigure}
     \begin{subfigure}[b]{0.49\textwidth}
         \centering
         \includegraphics[width=\linewidth]{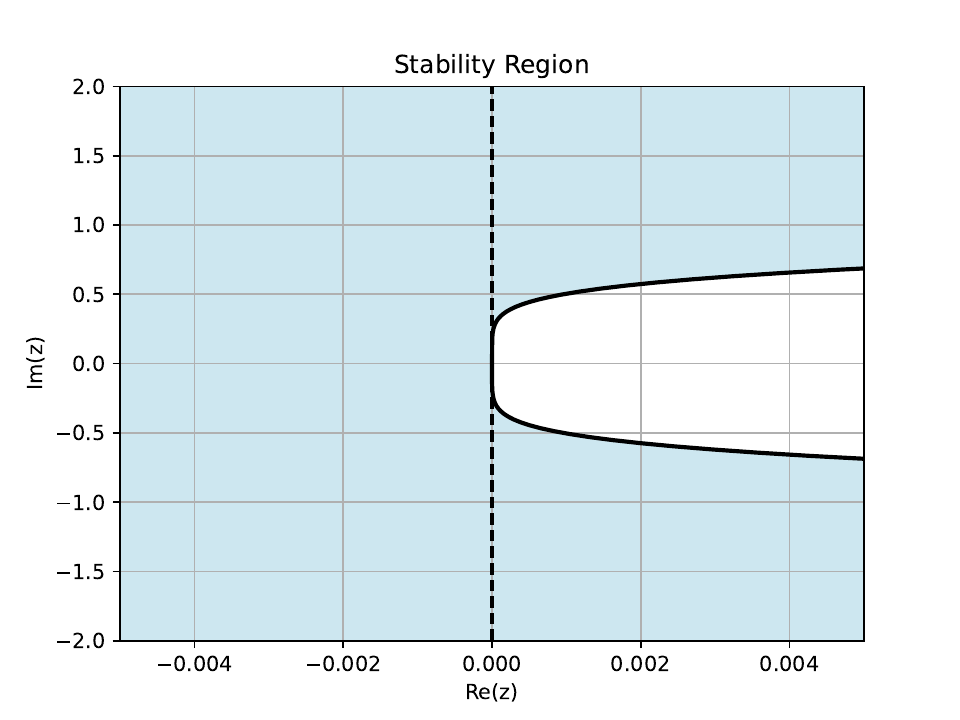}
     \end{subfigure}
     
     \caption{The region of absolute stability $\mathcal{S}$ (blue) for the optimal BDFL$3_2$ method (left), and $\partial\mathcal{S}$  near the imaginary axis (right), see Table \ref{tab:bdfl_345}}
     \label{fig:bdfl3}
\end{figure}


\begin{figure}
     \centering
     \begin{subfigure}[b]{0.49\textwidth}
         \centering
         \includegraphics[width=\linewidth]{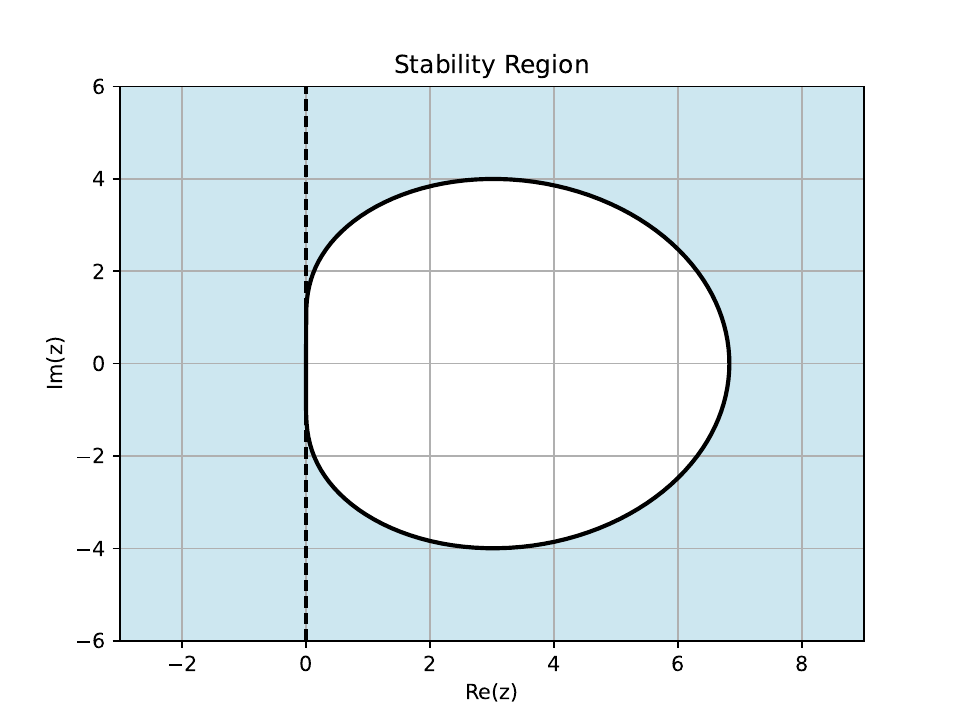}
     \end{subfigure}
     \begin{subfigure}[b]{0.49\textwidth}
         \centering
         \includegraphics[width=\linewidth]{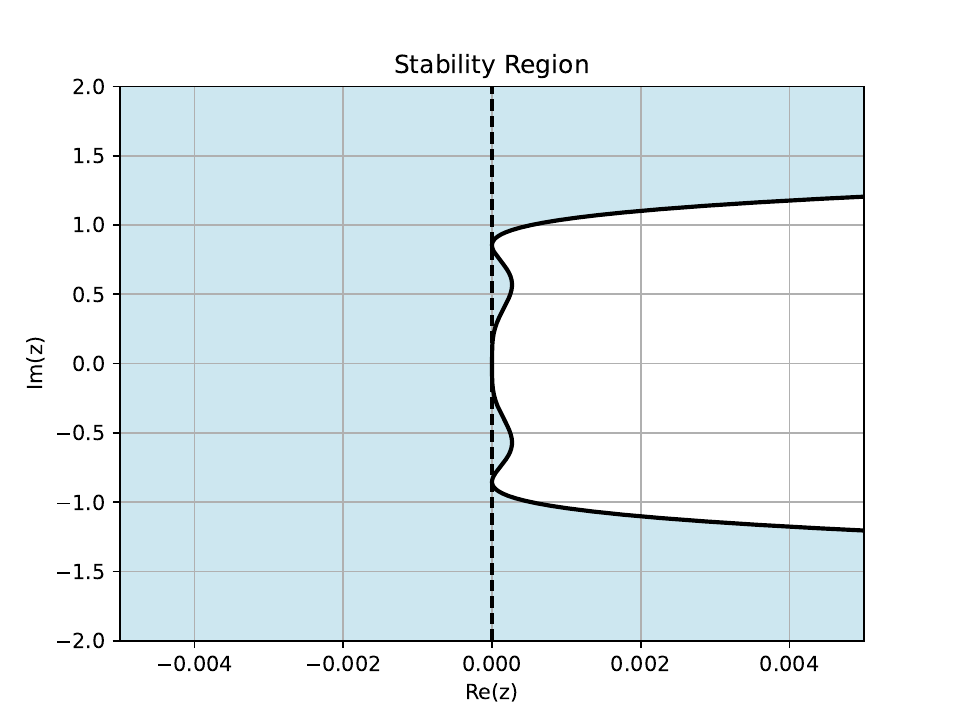}
     \end{subfigure}
     
     \caption{The region of absolute stability $\mathcal{S}$ (blue) for the optimal BDFL$4_2$ method (left), and $\partial\mathcal{S}$  near the imaginary axis (right), see Table \ref{tab:bdfl_345}}
     \label{fig:bdfl4}
\end{figure}



\begin{figure}
     \centering
     \begin{subfigure}[b]{0.49\textwidth}
         \centering
         \includegraphics[width=\linewidth]{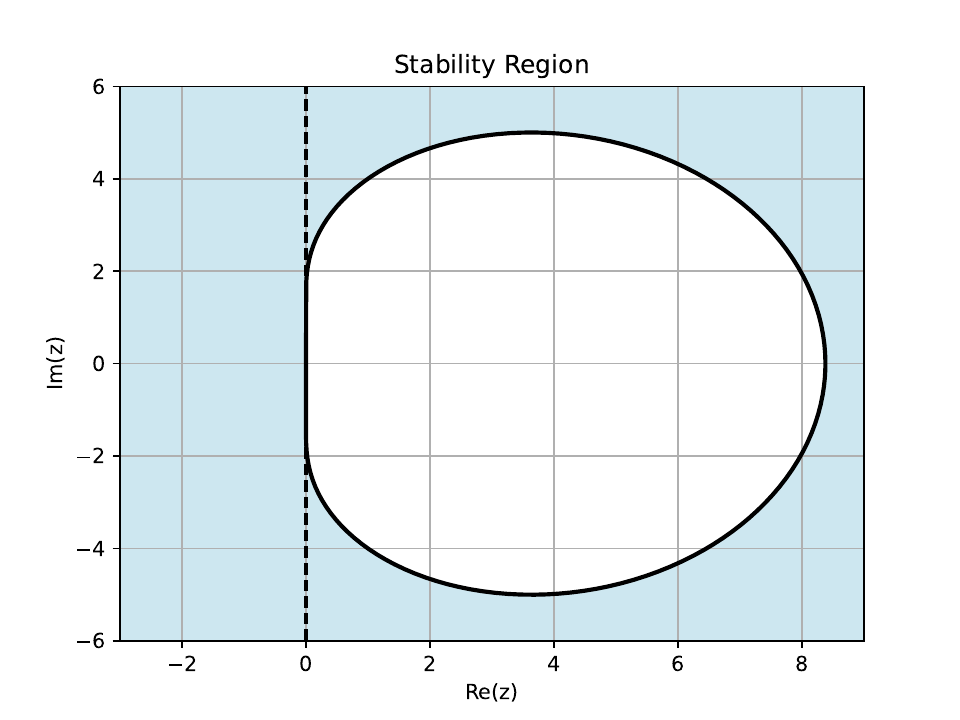}
     \end{subfigure}
     \begin{subfigure}[b]{0.49\textwidth}
         \centering
         \includegraphics[width=\linewidth]{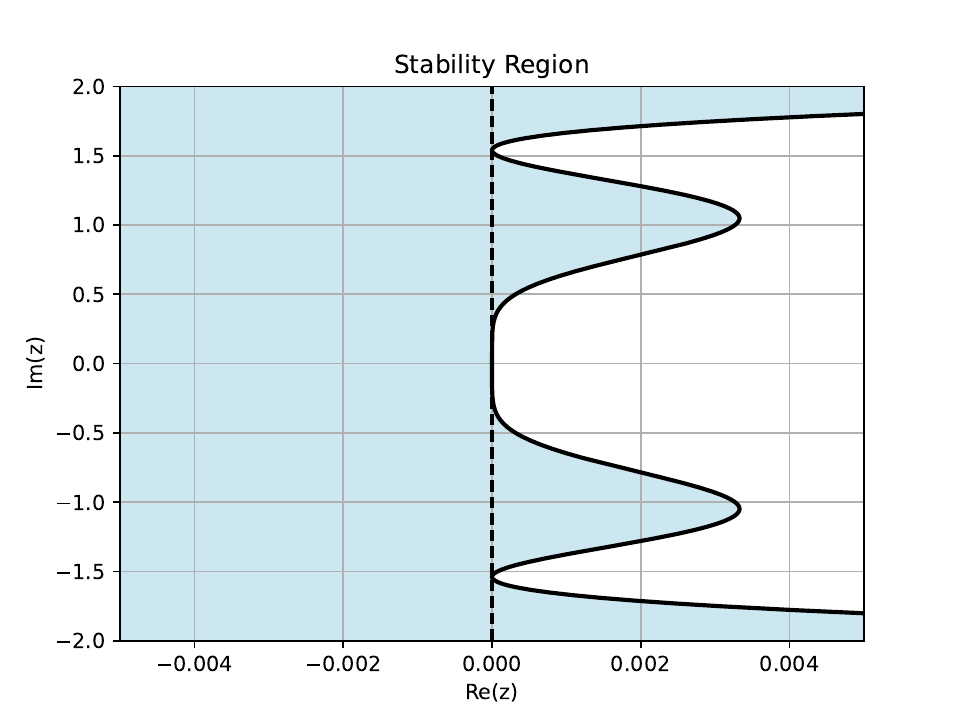}
     \end{subfigure}
     
     \caption{The region of absolute stability $\mathcal{S}$ (blue) for the essentially optimal BDFL$5_2$ method (left), and $\partial\mathcal{S}$  near the imaginary axis (right), see Table \ref{tab:bdfl_345}}
     \label{fig:bdfl5}
\end{figure}


\section{Optimal A-stable and L$(\kappa)$-stable methods}
\label{sec:Lkappa}

In Section \ref{sec:bdflike}, 
we studied methods with asymptotic damping factor
equal to zero.  In practice, for most problems good results can be
obtained with methods that have a non-zero (but less than unity) asymptotic
damping factor, since this leads to the elimination of stiff oscillations
within a few steps.
We now turn to the study of such methods, denoted as L($\kappa$)-stable
methods where $\kappa$ is the asymptotic damping factor.



\subsection{$2$-step $2^\text{nd}$-order methods}\label{sec:k=2strictlyLkappamethods}

We start by considering two-step methods of second order.  In this class, the unique L(0)-stable method is BDF2, while the method with smallest error coefficient is the trapezoidal method, which unfortunately is not even strictly L(1)-stable.
Thus we look for methods that have a smaller error coefficient than BDF2 and a smaller asymptotic damping factor than the trapezoidal method.



Taking into account the normalization \eqref{normalizationassumption}  and applying the order conditions \eqref{consistency}, we can express
the set of two-step second-order methods in terms of the two parameters $\beta_1, \beta_2$, with the remaining coefficients given by
\begin{align} \label{2stepconsistency}
    (\alpha_0,\alpha_1,\alpha_2,\beta_0)=(-2 \beta _1-4 \beta _2+3,  2 \beta _1+4 \beta _2-4,1,\beta _1+3 \beta _2-2).
\end{align}

Using Mathematicas's \verb|Resultant|, we find that the irreducibility assumption \eqref{irredasusumption} holds if and only if
\[
\left(2 \beta _2-1\right) \left(\beta _1+2 \beta _2-1\right)\neq 0.
\]

To find the conditions for A-stability, we use Algorithm \ref{Astabilitydeterminationalg}:
 the necessary and sufficient condition for irreducibility and A-stability in terms of the coefficients $\beta_{1,2}$ is
 \begin{equation}\label{2stepAstabilitycondition}
      \left(\beta _1\leq 0\land \frac{1}{2} \left(1-\beta _1\right)<\beta _2\leq \frac{1}{2} \left(2-\beta _1\right)\right)\lor
   \left(0<\beta _1<1\land \frac{1}{2}<\beta _2\leq \frac{1}{2} \left(2-\beta _1\right)\right).
 \end{equation}

\begin{figure}
    \centering
    \includegraphics[width=0.99\linewidth]{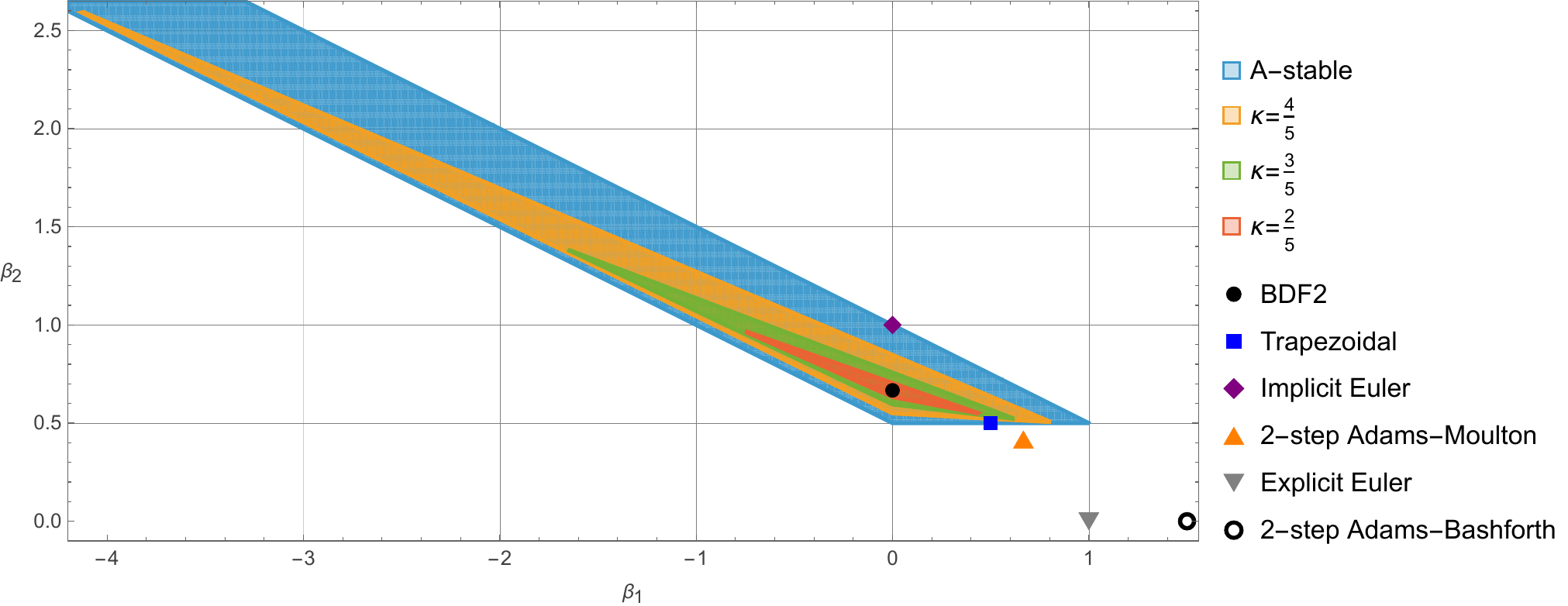}
    \caption{The unbounded set of irreducible $2$-step $2^\text{nd}$-order A-stable ILMMs \eqref{2stepAstabilitycondition}
    in blue (the rest of the coefficients are given in  \eqref{2stepconsistency}), 
    the bounded sets of 
    A-stable and strictly L($\kappa$)-stable ILMMs 
    for $\kappa\in\left\{\frac{4}{5},\frac{3}{5},\frac{2}{5}\right\}$, and some particular LMMs (see also Remark \ref{rem:1stepembedding}) 
    } 
    \label{fig:a_stable_2step}
\end{figure}

To investigate L($\kappa$)-stability we fix $\kappa\in(0,1)$ and apply Algorithm \ref{Schuralgorithm}, which yields that 
an ILMM with \eqref{2stepconsistency} is strictly L$(\kappa)$-stable
if and only if
\begin{equation}\label{2stepStrictL(kappa)condition}
 \left| \beta _1+3 \beta _2-2\right| < \left| \kappa ^2 \beta _2\right| \quad  \land \quad \left| \frac{\kappa  \beta _1}{\beta _2 \kappa
   ^2+\beta _1+3 \beta _2-2}\right| <1.  
\end{equation}
Figure \ref{fig:a_stable_2step} displays these regions for some values of $\kappa$.

\begin{rem}\label{rem:1stepembedding}
    In Figure \ref{fig:a_stable_2step} (and in some figures later), we have also displayed some ELMMs and $1$-step methods for comparison. For the purpose of these plots, the explicit and implicit Euler methods are interpreted as reducible $2$-step methods  $y_{n+2}-y_{n+1}=h f(y_{n+1})$ and  $y_{n+2}-y_{n+1}=h f(y_{n+2})$. The trapezoidal method \eqref{trapezoidalmethod} is analogously embedded in the family of 2-step methods.
\end{rem}


To determine the smallest error constant  within the set of A-stable and strictly L($\kappa$)-stable ILMMs, first we substitute the consistency rules \eqref{2stepconsistency} into \eqref{errorconstantmodulustobeminimized} and get
\begin{equation}\label{2stepObjectiveFunction}
     |C|=\frac{1}{12} \left| \frac{\beta_1+8 \beta_2-4}{\beta_1+2 \beta_2-1}\right|.
\end{equation}
Then we use \texttt{Minimize} with objective function \eqref{2stepObjectiveFunction} and constraints \eqref{2stepAstabilitycondition} and \eqref{2stepStrictL(kappa)condition}, which yields that the optimal $|C|$ as a function of $\kappa\in(0,1)$ is
\begin{equation}\label{2stepOptimalvalue}
   |C|_\text{opt}= \frac{\kappa ^2-\kappa +1}{3 (\kappa +1)^2},
\end{equation}
although \texttt{Minimize} indicates that the minimum is not attained, so 
\eqref{2stepOptimalvalue} is only an infimum.

\begin{figure}
    \centering
    \includegraphics[width=0.6\linewidth]{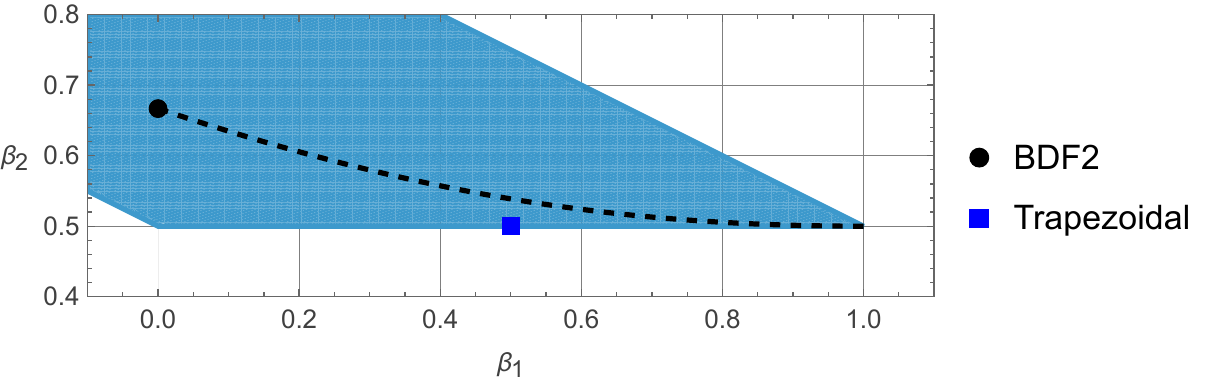}
    \caption{The set of optimal irreducible $2$-step $2^\text{nd}$-order A-stable and L($\kappa$)-stable ILMMs is given by the dashed black curve in the $(\beta_1,\beta_2)$ parameter plane, see Corollary \ref{cor:k=2Lkappaoptimalmethods}.
    }
    \label{fig:a_stable2_and_lk}
\end{figure}


Based on the above investigations, we can identify the optimal (non-strictly) L$(\kappa)$-stable methods---some of which will appear in Table \ref{tab:lk_1234}.

\begin{cor}\label{cor:k=2Lkappaoptimalmethods}
    For any $\kappa\in[0,1)$, the optimal irreducible $2$-step $2^\text{nd}$-order A-stable and L$(\kappa)$-stable ILMM  
    has $|C|_\emph{opt}=\frac{\kappa ^2-\kappa +1}{3 (\kappa +1)^2}$
    with coefficients
    \begin{subequations}\label{corollary43betacoeffs}
    \begin{align}
       (\alpha_0,\alpha_1,\alpha_2)=\left(\frac{3 \kappa -1}{\kappa -3},\frac{4 (1-\kappa )}{\kappa -3},1\right), \\
       (\beta_0,\beta_1,\beta_2)=\left(\frac{2 \kappa ^2}{-\kappa ^2+2 \kappa +3},\frac{4 \kappa }{-\kappa ^2+2 \kappa +3},\frac{2}{-\kappa ^2+2 \kappa +3}\right).
    \end{align}
    \end{subequations}
    The classical BDF2 method corresponds to $\kappa=0$, and the LMM with $\kappa=1$ would be reducible, see Figures \ref{fig:a_stable2_and_lk} and \ref{fig:lk_errmin}.
\end{cor}
\begin{proof}
    We apply Algorithm \ref{vonNeumannalgorithm}. On the one hand, \texttt{Reduce} shows that Case 1 of the algorithm cannot occur under the assumptions $\kappa\in[0,1]$, \eqref{2stepAstabilitycondition}, and $|C|=\frac{1}{12} \left| \frac{\beta _1+8 \beta _2-4}{\beta _1+2 \beta _2-1}\right| \le\frac{\kappa ^2-\kappa +1}{3 (\kappa +1)^2}=|C|_\text{opt}$. On the other hand, Case 2 under the assumptions $\kappa\in(0,1)$, \eqref{2stepAstabilitycondition}, and $|C| \le|C|_\text{opt}$ has a unique solution $(\beta_1,\beta_2)$ given in \eqref{corollary43betacoeffs}, from which we recover the rest of the coefficients via \eqref{2stepconsistency}. The cases $\kappa\in\{0,1\}$ are checked separately.  We remark that for \eqref{corollary43betacoeffs} the second characteristic
polynomial $\sigma$ has a double root at $-\kappa$ (cf.\
Remark~\ref{rem:ZZMLMS3}).
\end{proof}


\begin{rem}\label{rem:ZZM2step}
The family \eqref{corollary43betacoeffs}, which we derived independently,
was obtained earlier in \cite{masarati2001multistep,zhang2020astable} and appears as the LMS2 method
of \cite{zhang2021improved} with $\rho_\infty=\kappa$ (their coefficients
$\alpha_j$ and $\beta_j$ correspond to our $-\alpha_{2-j}$ and
$\beta_{2-j}$). In \cite{zhang2020astable}, it was identified as the
optimal A-stable 2-step $2^\text{nd}$-order family for a given
$\rho_\infty$, based on closed-form expressions of the error constant on
subregions of the A-stable region, whose derivation was omitted, and on a
graphical comparison. Corollary~\ref{cor:k=2Lkappaoptimalmethods} provides
a complete and rigorous proof that, for every $\kappa\in[0,1)$, this family
is the unique optimum among all irreducible $2$-step $2^\text{nd}$-order
A-stable and L$(\kappa)$-stable ILMMs.
\end{rem}

\subsection{Some optimal $3$-step $2^\text{nd}$-order A-stable and L$(\kappa)$-stable  methods}\label{sec:k=3someLkappamethods}

Exploring the space of 3-step $2^\text{nd}$-order A-stable ILMMs is computationally more demanding, since the relevant parameter spaces have an additional
dimension. Therefore, in this section, we impose certain structural conditions (see Remark \ref{rem:structuralconditions}) to restrict the search space.



\subsubsection{Case 1: $\beta_0=\beta_1=0$}\label{sec:beta0beta1ZERO}

 The  normalization and $2^{\text{nd}}$-order consistency conditions \eqref{normalizationassumption},  \eqref{implicitasusumption} and \eqref{consistency}
imply that the two-parameter family of ILMMs now takes the form 
\begin{equation}\label{beta0beta1ZEROconsistency}
 (\alpha_0,\alpha_1,\alpha_2,\alpha_3)=\left(\frac{\beta
   _2}{2}+\frac{3 \beta _3}{2}-1,-2 \beta _2-4 \beta _3+3,\frac{3 \beta _2}{2}+\frac{5 \beta _3}{2}-3,1\right)
\end{equation}
with $\beta_2\in\mathbb{R}$ and $\beta_3\in\mathbb{R}\setminus\{0\}$.

Irreducibility \eqref{irredasusumption} of the ILMM is equivalent to
$
\left(\beta _2+\beta _3\right) \left(3 \beta _3-2\right) \left(\beta _2+3 \beta _3-2\right)\neq 0$.

 Checking A-stability in this family is computationally non-trivial, even when using 
Algorithm \ref{Astabilitydeterminationalg} with
the simpler substitution \eqref{simplersubstitution} instead of the Weierstrass substitution. It turns out that A-stability with irreducibility is equivalent to
\begin{equation}\label{AstabirredinFamilybeta0beta1ZERO}
   \left[ \left(\frac{1}{2}<\beta _3<\frac{2}{3}\ \ \land \ \ \frac{1}{2} \left(2-5 \beta _3\right)+\frac{1}{2} \sqrt{-15 \beta _3^2+4 \beta
   _3+4}\leq \beta _2\leq \beta _3\right)\lor \right.
\end{equation}
\[
\left.\left(\frac{2}{3}<\beta _3<2\ \ \land\ \  -\beta _3\leq \beta _2\leq 2-2 \beta _3\right)\right]\ \ \land\ \ \left(\beta _2+\beta _3\right)  \left(\beta _2+3 \beta _3-2\right)\neq 0.
\]
We remark that the smallest closed horizontal-vertical rectangle in the $(\beta_2,\beta_3)$ parameter plane containing the above set is  $\left[-2,\frac{2}{3}\right]\times\left[\frac{1}{2},2\right]$.

 This time, the second characteristic polynomial $\sigma$ admits a trivial factorization, so it is directly seen that 
\begin{equation}\label{kappastabinFamilybeta0beta1ZERO}
\text{the LMM is L}(\kappa)\text{-stable for some  }\kappa\ge 0\quad \Longleftrightarrow \quad \left| \beta _2\right| \le \kappa  \left| \beta _3\right|.
\end{equation}
Figure \ref{fig:a_stable_b0b1} shows the set of irreducible A-stable and  L$(\kappa)$-stable methods for some values of $\kappa$. 
\begin{figure}
    \centering
    \includegraphics[width=0.75\linewidth]{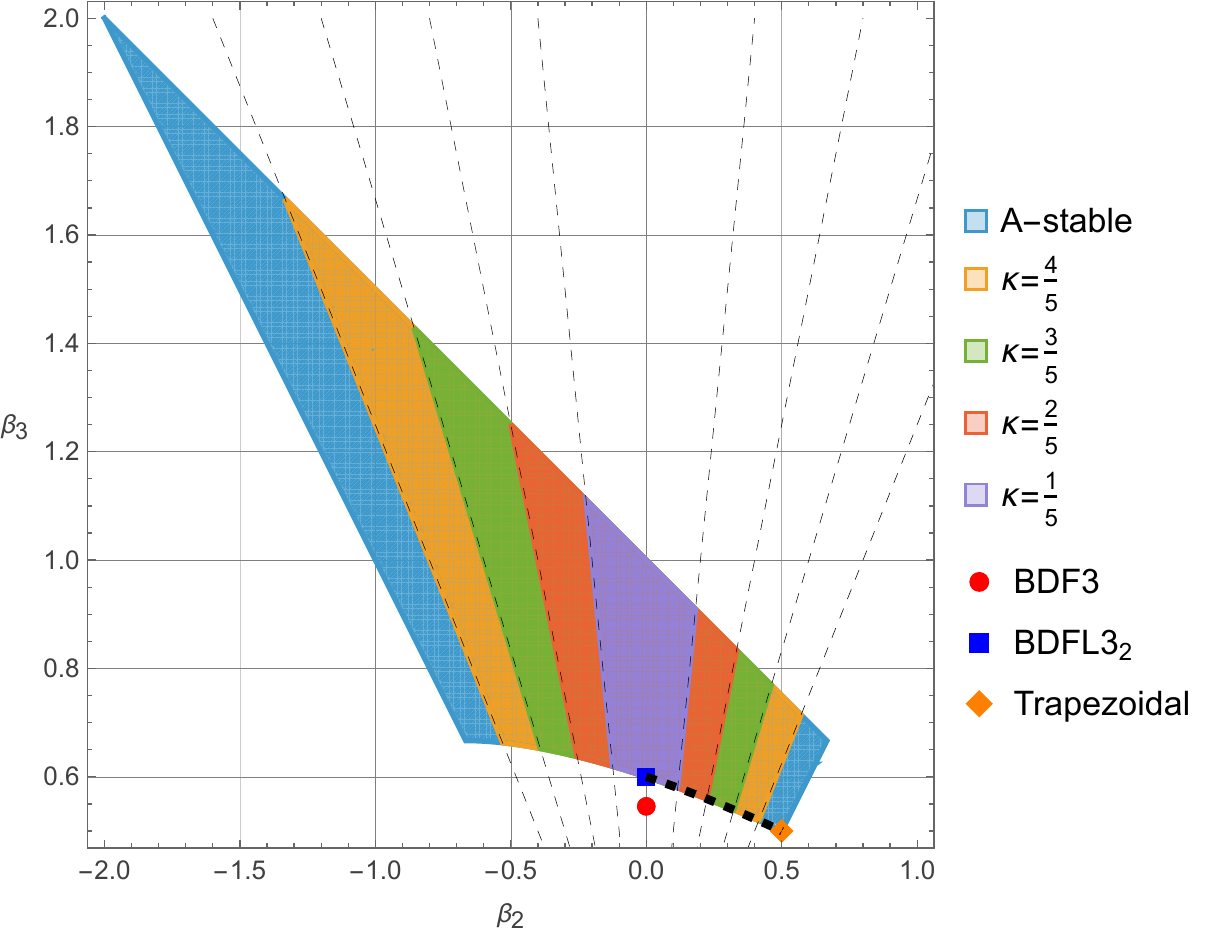}
    \caption{The set of irreducible $3$-step $2^\text{nd}$-order A-stable  and L$(\kappa)$-stable ILMMs in the $(\beta_2,\beta_3)$ parameter plane when $\beta_0=\beta_1=0$ for some values of $\kappa$, see \eqref{AstabirredinFamilybeta0beta1ZERO} and \eqref{kappastabinFamilybeta0beta1ZERO}.
    The classical BDF3 and the trapezoidal method (embedded in this family as a reducible 3-step method in the sense of Remark \ref{rem:1stepembedding}) are also displayed for comparison. The optimal BDFL3$_2$ method (Table \ref{tab:bdfl_345}) is on the boundary and in the set, the trapezoidal method is on the boundary but not in the set, and the BDF3 method lies outside the set. The dashed gray rays emanating from the origin form the boundaries of the sectors given by \eqref{kappastabinFamilybeta0beta1ZERO}. The dashed black curve on the lower right boundary of the set 
    represents the optimal L$(\kappa)$-stable methods, see Corollary \ref{cor:beta0beta1ZEROoptimalmethods}.
    }
    \label{fig:a_stable_b0b1}
\end{figure}

 The modulus of the error constant is 
\begin{equation}\label{beta0beta1ZEROCvalue}
    |C|=\frac{1}{6} \left| \frac{2 \beta _2+11 \beta _3-6}{\beta _2+\beta _3}\right|.
\end{equation}

In the following, we identify the optimal L$(\kappa)$-stable methods; some of them will be presented in Table \ref{tab:lk_1234}.


\begin{cor}\label{cor:beta0beta1ZEROoptimalmethods}
    Let us assume $\beta_0=\beta_1=0$, and choose some $\kappa\in[0,1)$. Then the optimal irreducible $3$-step $2^\text{nd}$-order A-stable and  L$(\kappa)$-stable ILMM  
    has 
    \begin{equation}\label{corollary51Copt}
    |C|_\emph{opt}=\frac{3-\kappa }{6 (\kappa +3)}
    \end{equation}
    with coefficients
    \[
       (\alpha_0,\alpha_1,\alpha_2,\alpha_3)=\left(\frac{\kappa -1}{\kappa ^2+5 \kappa +10},\frac{-\kappa ^2-5 \kappa +6}{\kappa ^2+5 \kappa +10},\frac{-\kappa -15}{\kappa ^2+5 \kappa +10},1\right),
    \]
    \begin{equation}\label{corollary51betacoeffs}
       (\beta_0,\beta_1,\beta_2,\beta_3)=\left(0,0,\frac{2\kappa \left(\kappa +3 \right)}{\kappa ^2+5 \kappa +10},\frac{2 (\kappa +3)}{\kappa ^2+5 \kappa +10}\right);
    \end{equation}
    see also Figures \ref{fig:a_stable_b0b1} and \ref{fig:lk_errmin}.
\end{cor}
\begin{proof} 
 We let \texttt{Minimize} run for several hours but---unlike in the case of \eqref{2stepOptimalvalue}---it could not determine the optimal value of $|C|$ when $\kappa$ was present as a parameter. So we sampled some suitable concrete values of $\kappa\in(0,1)$, and now \texttt{Minimize} quickly returned an answer. Based on these data points, we conjectured \eqref{corollary51Copt}. To actually prove Corollary \ref{cor:beta0beta1ZEROoptimalmethods}, we use \texttt{Reduce}, which shows that the system formed by the inequalities 
 \begin{itemize}
     \item $0<\kappa<1$, \eqref{AstabirredinFamilybeta0beta1ZERO},
     \item $\left| \beta _2\right| \le \kappa  \left| \beta _3\right|$ (see \eqref{kappastabinFamilybeta0beta1ZERO}), and 
     $\frac{1}{6} \left| \frac{2 \beta _2+11 \beta _3-6}{\beta _2+\beta _3}\right|\le \frac{3-\kappa }{6 (\kappa +3)}$ (see \eqref{beta0beta1ZEROCvalue})
 \end{itemize}
   has a unique solution $(\beta_2,\beta_3)$ given by \eqref{corollary51betacoeffs}. The remaining coefficients are then obtained from \eqref{beta0beta1ZEROconsistency}.
The case $\kappa=0$ is checked separately; as expected, we recover the BDFL$3_2$ method presented in Table \ref{tab:bdfl_345}.
We remark that for $\kappa=1$ we would get a reducible LMM  equivalent to the trapezoidal method. 
\end{proof}

\subsubsection{Case 2: $\beta_0=\beta_2=0$}\label{sec:beta0beta2ZERO}


 Using the consistency and normalization conditions, we express the remaining method coefficients with $(\beta_1,\beta_3)$.

 When applying Algorithm \ref{Astabilitydeterminationalg} directly---either with the Weierstrass substitution or with the cosine substitution \eqref{simplersubstitution}---we could not obtain a description of the set of A-stable  methods $\mathcal{A}$ with \texttt{Reduce} in a reasonable amount of computing time. Therefore, we first try and employ  some necessary conditions again. Let \texttt{rationalFunction} denote the function obtained by the cosine substitution; its numerator is a quintic polynomial in $r$ and a cubic polynomial in $c$, while its denominator is quintic in $r$ and quadratic in $c$ (with the Weierstrass substitution, the corresponding $(r,t)$ degrees in the numerator and denominator would be $(5,6)$ and $(5,6)$). Instead of requiring the positivity of \texttt{rationalFunction} for any $r>1$ and $c\in[-1,1]$, we just require 
\[
\texttt{rationalFunction}>0 \quad \forall r>1\quad\text{and}\quad c\in\{-1,0,1\}
\]
and
\[
\texttt{rationalFunction}\ge 0 \quad \text{for } r=1\quad\text{and}\quad \forall c\in[-1,1].
\]
\texttt{Reduce} then eliminates the quantifiers and yields some non-linear inequalities in $(\beta_1,\beta_3)$; let $\mathcal{P}$ denote the intersection of the corresponding regions. The set $\mathcal{P}$ has a simple description, and 
we have $\mathcal{A}\subseteq\mathcal{P}\subseteq [-1,1]\times\left[\frac{1}{2},1\right]$ in the $(\beta_1,\beta_3)$ plane.
Now, interestingly, it turns out that these necessary conditions are also sufficient: \texttt{Reduce} establishes (within a few minutes) that the system formed by the inequalities describing $\mathcal{P}$,  $r>1$, $-1\le c\le 1$,  and \texttt{rationalFunction}$\,\le 0$ has no solution. Therefore, $\mathcal{A}=\mathcal{P}$, see Figure \ref{fig:supersetandminerror_b0b2}.

 The second characteristic polynomial $\sigma$ admits a simple factorization, so  
\begin{equation}\label{kappastabinFamilybeta0beta2ZERO}
\text{the LMM is L}(\kappa)\text{-stable for some  }\kappa\ge 0\quad \Longleftrightarrow \quad \left| \beta _1\right| \le \kappa^2  \beta _3.
\end{equation}

 Finally, it is easily seen that, for fixed $\kappa\in[0,1)$, the error constant $|C|$ attains its minimum value at the lower curved boundary of $\mathcal{A}$ for $\beta_1\in[-1/2,0]$ 
(see Figure \ref{fig:supersetandminerror_b0b2} again). In this way, we have obtained the corresponding graph in Figure \ref{fig:lk_errmin}, describing the optimal error constant 
$|C|_\text{opt}=\frac{3-\kappa^2 }{6 (\kappa^2 +3)}$
 (cf.~\eqref{corollary51Copt}).

\subsubsection{Case 3: $\beta_1=\beta_2=0$}

Using a similar analysis as above, we have found that the optimal irreducible $3$-step $2^\text{nd}$-order A-stable and  L$(\kappa)$-stable ILMMs  
 have the common error constant $|C|_\text{opt}=1/6$ for any $\kappa\in[0,1]$, see Figure \ref{fig:lk_errmin}.

\begin{figure}
     \centering
     \begin{subfigure}[b]{0.54\textwidth}
        \centering
    \includegraphics[width=1\linewidth]{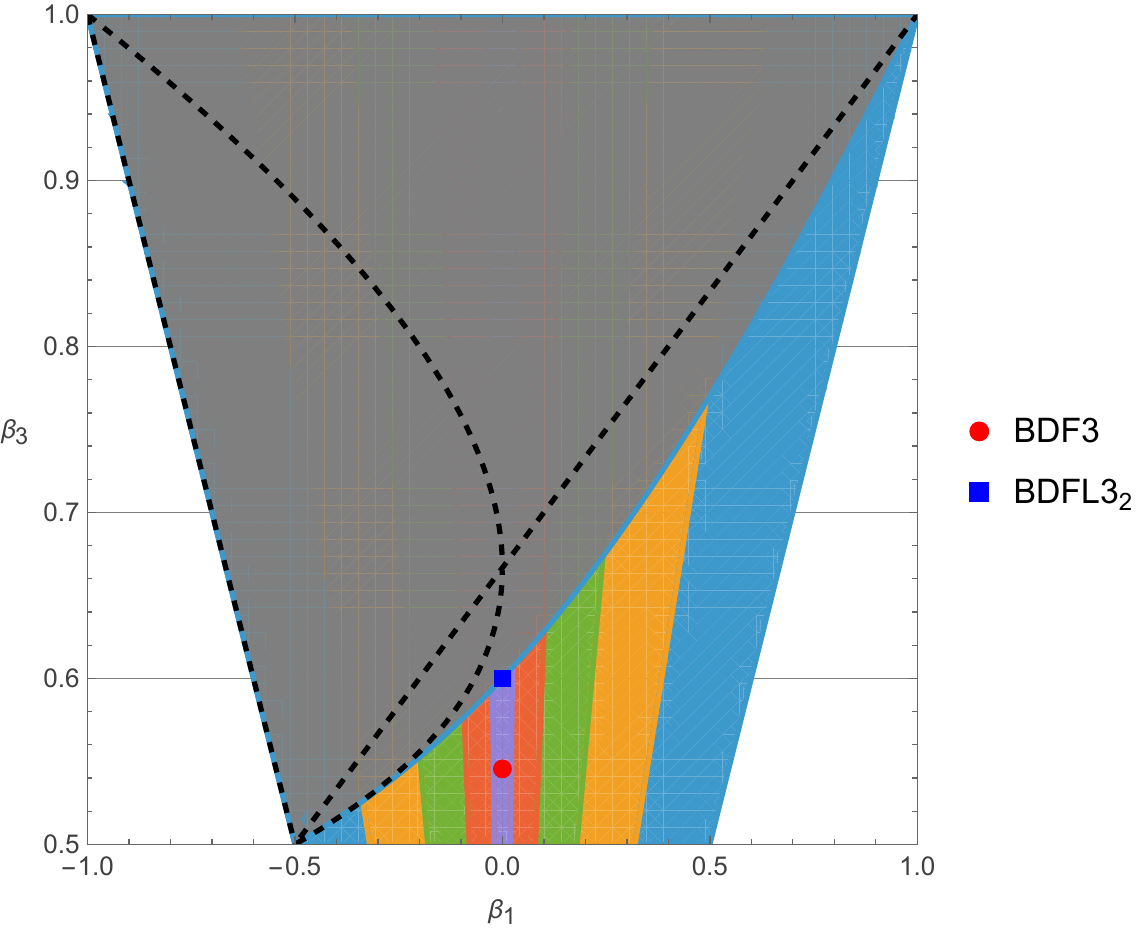}
     \end{subfigure}
     \begin{subfigure}[b]{0.44\textwidth}
         \centering
         \includegraphics[width=1\linewidth]{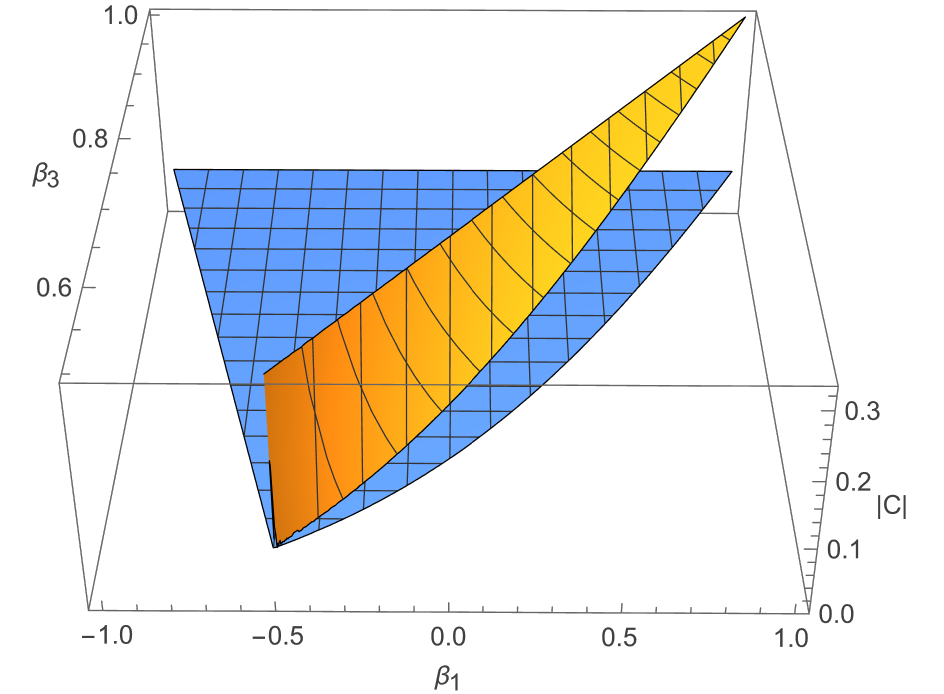}
     \end{subfigure}
     
     \caption{The set of 3-step A-stable LMMs when $\beta_0=\beta_2=0$ is colored gray in the left figure. Reducible methods are denoted by the dashed black curves. The sets of L$(\kappa)$-stable methods are shown in various colors for some values of $\kappa\in(0,1)$---as $\kappa$ decreases, the overlapping wedges become narrower, see \eqref{kappastabinFamilybeta0beta2ZERO}. The yellow surface in the right figure shows the error constant $|C|$ (values greater than $1/3$ are not displayed), while the blue constant surface represents the theoretical minimum at height $1/12$.}
     \label{fig:supersetandminerror_b0b2}
\end{figure}



\begin{figure}
    \centering
    \includegraphics[width=0.75\linewidth]{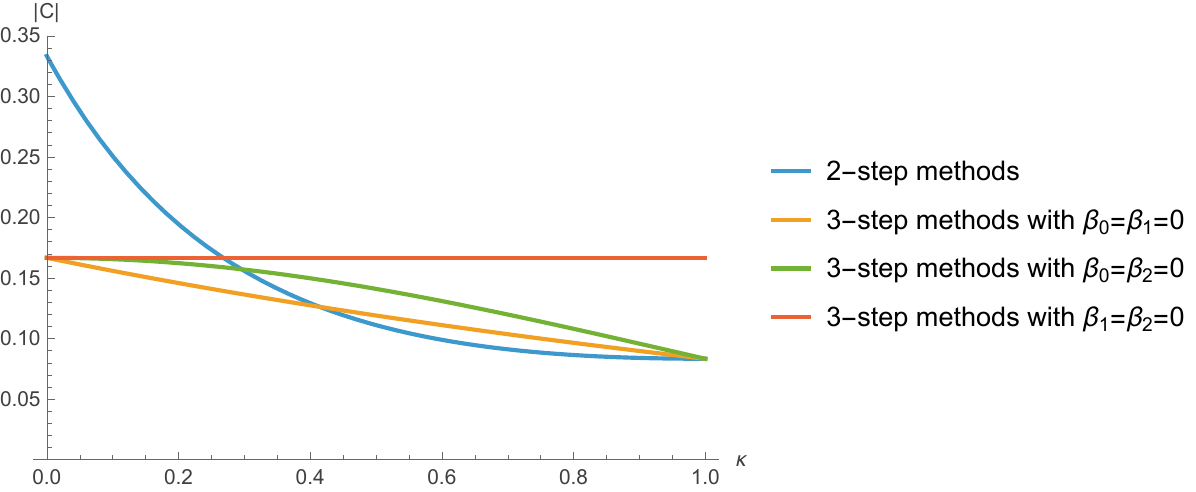}
    \caption{The smallest error constants \eqref{errorconstantmodulustobeminimized} in the family of $2$-step and in some families of $3$-step irreducible $2^\text{nd}$-order A-stable and L$(\kappa)$-stable ILMMs 
    as a function of $\kappa$, summarizing our findings in Sections \ref{sec:k=2strictlyLkappamethods}--\ref{sec:k=3someLkappamethods}. 
    At the left endpoint $\kappa=0$, the values are given by the classical BDF2 and the optimal BDFL$3_2$ methods ($1/3$ and $1/6$, respectively), while at the right endpoint $\kappa=1$ by the trapezoidal method ($1/12$). 
    }
    \label{fig:lk_errmin}
\end{figure}

\subsection{Some optimized methods}
In Table \ref{tab:lk_1234}, we summarize the properties of some optimal two- and three-step methods found in Sections \ref{sec:k=2strictlyLkappamethods}--\ref{sec:k=3someLkappamethods}.  For methods of a given step number, we clearly see a tradeoff between reducing $\kappa$ and reducing $|C|$.

\begin{table}
\centering
\renewcommand{\arraystretch}{1.4}
\begin{tabularx}{\textwidth}{c|X|c}
\textbf{$\kappa$ value} & \textbf{LMM} & \textbf{Error constant $C$} \\
\hline
$\frac{1}{5}$ &
$y_{n+3}=\frac{95}{69} y_{n+2}-\frac{31}{69}y_{n+1}+\frac{5}{69}y_n+\frac{h}{69} \left(8 f(y_{n+2})+40
   f(y_{n+3})\right)$ &
$-\frac{7}{48}\approx -0.14583$\\[6pt]
\hline
$\frac{2}{5}$ &
$y_{n+3}=\frac{385}{304} y_{n+2}-\frac{6}{19}y_{n+1}+\frac{15}{304}y_n+$  $h \big(\frac{17}{76} f(y_{n+2}) +\frac{85}{152}
   f(y_{n+3})\big)$ &
$-\frac{13}{102}\approx -0.12745$\\[6pt]
\hline
$\frac{3}{5}$ &
$y_{n+2}=\frac{2}{3}y_{n+1}+\frac{1}{3}y_n+h \big(\frac{3}{16} f(y_n)+\frac{5}{8} f(y_{n+1})+\frac{25}{48}
   f(y_{n+2})\big)$ &
$-\frac{19}{192}\approx -0.09896$ \\[6pt]
\hline
$\frac{4}{5}$ &
$y_{n+2}=\frac{4}{11} y_{n+1}+$ $\frac{7}{11}y_n+$ $\frac{h}{99} \big(32 f(y_n)+80 f(y_{n+1})+50
   f(y_{n+2})\big)$ &
$-\frac{7}{81}\approx -0.08642$\\
\end{tabularx}
\caption{Selected irreducible $2^\text{nd}$-order A-stable and L$(\kappa)$-stable methods and their associated error constants (see Corollaries
\ref{cor:k=2Lkappaoptimalmethods} and \ref{cor:beta0beta1ZEROoptimalmethods}).
For each value of $\kappa$, the corresponding method has the
smallest attainable value of the error constant $|C|$ among the classes of methods studied herein.
For $\kappa\in\{1/5, 2/5\}$, the optimal methods have three steps, while for $\kappa \in \{3/5,4/5\}$ the optimal methods
use only two steps. A 3-step family from \cite{zhang2021improved} with smaller $|C|$ for
every $\kappa\in(0,1)$, lying outside these classes, is discussed in
Remark~\ref{rem:ZZMLMS3}.
}
\label{tab:lk_1234}
\end{table}
\begin{rem}\label{rem:ZZMLMS3}
In \cite{zhang2021improved}, a 3-step family (LMS3) was obtained by
requiring $\sigma$ to have a triple root at $-\rho_\infty$ and minimizing
$|C|$ in the resulting one-parameter family. With $\rho_\infty=\kappa$, its
error constant is $|C|=(1+\kappa^2)/\bigl(6(1+\kappa)^2\bigr)$. For
$\kappa=0$ it reduces to BDFL$3_2$, but for $\kappa\in(0,1)$ it lies outside
the three classes considered in Section~\ref{sec:k=3someLkappamethods} and
has a smaller error constant than all methods found there: for example,
$|C|\approx 0.1204$ for $\kappa=1/5$ and $|C|\approx 0.0986$ for
$\kappa=2/5$, compared with $0.1458$ and $0.1275$ in
Table~\ref{tab:lk_1234}. Its error constant is also smaller than that of the
optimal 2-step method of Corollary~\ref{cor:k=2Lkappaoptimalmethods}, by
$(1-\kappa)^2/\bigl(6(1+\kappa)^2\bigr)$. Using
Algorithm~\ref{Astabilitydeterminationalg}, we have certified that LMS3 is
A-stable for every $\kappa\in[0,1)$ (in \cite{zhang2021improved}, this was
verified analytically on the imaginary axis and numerically elsewhere).
Moreover, among all 3-step methods with $\sigma(\zeta)=c\,(\zeta+\kappa)^3$,
LMS3 is the unique minimizer of $|C|$: for smaller $c$, one has
$\mathrm{Re}\left((\rho/\sigma)(e^{\imath\theta})\right)<0$ for small $\theta\ne0$,
violating Remark~\ref{rem:Astabneccond}, while $|C|$ increases with $c$.
We conjecture that, for every $k\ge 2$ and $\kappa\in[0,1)$, the optimal
$k$-step $2^\text{nd}$-order A-stable and L$(\kappa)$-stable methods
satisfy $\sigma(\zeta)=c\,(\zeta+\kappa)^k$; this holds for $k=2$ by
Corollary~\ref{cor:k=2Lkappaoptimalmethods}, and trivially for $\kappa=0$.
\end{rem}

\subsection{Phase error, low-frequency dissipation, and transient overshoot}
\label{sec:spectral}

In structural dynamics, time integrators are customarily assessed on the
undamped oscillator $y''+\omega^2 y=0$ by their \textit{relative period error}
(phase accuracy), their \textit{algorithmic damping} in the low- and
high-frequency ranges, and their tendency to \textit{overshoot}, i.e., to
produce a response that temporarily grows beyond the initial data during the
first steps \cite{HHT1977,HilberHughes1978,Hughes1987,ChungHulbert1993}; see also \cite{zhang2021improved}, where the amplitude decay and period
elongation of LMS2--LMS4 were computed numerically. Rewriting the oscillator as a first-order
system $u'=Au$, where $A$ has the eigenvalues $\pm\imath\omega$, and diagonalizing $A$ reduces it to the Dahlquist test
equation $u'=\lambda u$ with $\lambda=\pm\imath\omega$.
Since our LMMs have real coefficients,
the roots of the stability polynomial $\pi(\cdot\,;z)$ (see \eqref{stabfun})
at $z=-\imath\Omega$ are the complex conjugates of those at $z=\imath\Omega$,
where $\Omega:=\omega h>0$; hence the above properties are governed by the
roots of $\pi(\cdot\,;\imath\Omega)$. Below we relate the above three performance
indicators to the two key quantities of this work: the error constant $C$,
whose modulus we minimize, and $R_\infty$.

Let $\zeta_1(z)$ denote the principal root of $\pi(\cdot\,;z)$, i.e., the
root with $\zeta_1(0)=1$, which is analytic near $z=0$. For a
$2^\text{nd}$-order LMM with real coefficients and error constant $C$ in
Definition~\ref{def:errorconstant}, one has
\begin{equation}\label{eq:principal_root}
  \ln\zeta_1(z) = z - C z^3 - D z^4 + \mathcal{O}(z^5) \qquad (z\to 0)
\end{equation}
for some $D\in\mathbb{R}$. At the point $z=\imath\Omega$ of the imaginary
axis, we write the amplification factor in polar form as
$\zeta_1(\imath\Omega)=|\zeta_1(\imath\Omega)|\,e^{\imath\widetilde{\Omega}}$,
where $\widetilde{\Omega}$ is the numerical phase angle per time step. The
relative period error and the algorithmic damping ratio
\cite[Section~9.1.3]{Hughes1987} are then defined as
\[
  \operatorname{Err}_{\text{per}}(\Omega) := \frac{\Omega}{\widetilde{\Omega}}-1,
  \qquad
  \widetilde{\xi}(\Omega) := -\frac{\ln|\zeta_1(\imath\Omega)|}{\widetilde{\Omega}}.
\]
Taking the imaginary and real parts of \eqref{eq:principal_root} at
$z=\imath\Omega$ gives $\widetilde{\Omega}=\Omega+C\Omega^3+\mathcal{O}(\Omega^5)$
and $\ln|\zeta_1(\imath\Omega)|=-D\Omega^4+\mathcal{O}(\Omega^6)$, hence
\[
  \operatorname{Err}_{\text{per}}(\Omega) = -C\,\Omega^2 + \mathcal{O}(\Omega^4),
  \qquad
  \widetilde{\xi}(\Omega) = D\,\Omega^3 + \mathcal{O}(\Omega^5).
\]
Thus, the leading coefficient of the relative period error is $-C$. By
Dahlquist's second barrier, every A-stable $2^\text{nd}$-order LMM satisfies
$C\le -1/12$ \cite[Theorem~V.1.4]{wanner1996solving}, so all such
methods---in particular, those in
Tables~\ref{tab:bdfl_345}--\ref{tab:lk_1234}---lengthen the period, by a
relative amount of approximately $|C|\,\Omega^2$. Consequently, minimizing
$|C|$ is, to leading order, equivalent to minimizing the low-frequency
relative period error.

Low-frequency dissipation (or damping) is a higher-order effect, governed
by $D$ rather than by $|C|$ or $R_\infty$, but its sign is fixed by
A-stability: since $|\zeta_1(\imath\Omega)|\le 1$ for every A-stable method,
we have $\widetilde\xi(\Omega)\ge 0$, hence $D\ge 0$ for every A-stable
$2^\text{nd}$-order LMM. The coefficient $D$ is conveniently computed along
the root-locus curve, where $\ln\zeta_1(z)=\imath\theta$: inverting
\eqref{eq:principal_root} yields
\[
  \mathrm{Re}\,\frac{\rho(e^{\imath\theta})}{\sigma(e^{\imath\theta})}
  = D\,\theta^4 + \mathcal{O}(\theta^6) \qquad (\theta\to 0).
\]
The resulting values are listed in Table~\ref{tab:dissipation_D}. All
optimized methods are considerably less dissipative than BDF2 in the
low-frequency range, and several of them have $D=0$, i.e.,
$\widetilde\xi(\Omega)=\mathcal{O}(\Omega^5)$; the damping of high-frequency
modes, in turn, is controlled by $R_\infty$.

\begin{table}[tbp]
  \centering
  \renewcommand{\arraystretch}{1.3}
  \caption{Leading coefficient $D$ of the algorithmic damping ratio
    $\widetilde\xi(\Omega)=D\,\Omega^3+\mathcal{O}(\Omega^5)$.}
  \label{tab:dissipation_D}
  \begin{tabular}{|c|c|c|}
    \hline
    \textbf{LMM} & $D$ & \textbf{Special cases} \\
    \hline
    BDF2 & $1/4$ & \\
    \hline
    optimal 2-step family \eqref{corollary43betacoeffs}
      & $(1-\kappa)^3/\bigl(4(1+\kappa)^3\bigr)$
      & \begin{tabular}[t]{@{}l@{}}
          $D=1/256$ for $\kappa=3/5$\\
          $D=1/2916$ for $\kappa=4/5$
        \end{tabular} \\
    \hline
    BDFL$4_2$ & $5/4-7\sqrt{2}/8\approx 0.0126$ & \\
    \hline
    BDFL$3_2$, BDFL$5_2$, and the & \multirow{2}{*}{$0$}
      & \multirow{2}{*}{$\kappa\in\{1/5,2/5\}$ from Table~\ref{tab:lk_1234}} \\
    optimal 3-step family \eqref{corollary51betacoeffs} & & \\
    \hline
  \end{tabular}
\end{table}
Transient overshoot is related to the structure of the roots of $\sigma$
rather than to $R_\infty$ alone. In the stiff limit $z\to\infty$, the
deviations $\delta_n$ of the numerical solution from the slow manifold
satisfy, to leading order, the recurrence
$\sum_{j=0}^{k}\beta_j\,\delta_{n+j}=0$ (cf.~Proposition \ref{LMMdampingfactorprop}), with the
deviations of the $k$ starting values as initial data. For the BDFL methods,
$\sigma(\zeta)=\beta_k\zeta^k$, so $\delta_n=0$ for $n\ge k$ and no
transient growth occurs. For the 3-step family \eqref{corollary51betacoeffs}, $\sigma$ has
the simple root $-\kappa$ besides a double root at $0$, so the deviations
decay monotonically in modulus. For the 2-step family \eqref{corollary43betacoeffs},
however, $\sigma$ has the double root $-\kappa$, so
$\delta_n=(a+bn)(-\kappa)^n$, and the deviations can grow before decaying.
This is the multistep counterpart of the overshoot phenomenon studied in
\cite{HilberHughes1978}, and it agrees with
\cite[Section~5.6]{zhang2020astable}, where overshoot of this family was
observed for $\rho_\infty$ close to $1$ and found to depend on the initial
data. It is also consistent with \cite[Section~4.2]{zhang2021improved},
where no overshoot was found at step $k$ for the start-up procedures
considered there: for those starting values the deviations decay
monotonically, whereas transient growth occurs, e.g., when
$\delta_1\approx 0$, as with an L-stable starting method. Overshoot is
visible in Figures~\ref{fig:lk_perf}--\ref{fig:lk_err}: in the
Prothero--Robinson test, $y_0$ deviates from the slow manifold by $1/2$,
while $y_1$ (computed by the L-stable TR-BDF2 method) is close to it, and
the deviation of the $\kappa=4/5$ method grows again after the first step,
temporarily exceeding its initial value before decaying. Hence, if the
initial data are far from the slow manifold, the BDFL methods, the 3-step
family \eqref{corollary51betacoeffs}, or the 2-step family
\eqref{corollary43betacoeffs} with small $\kappa$ are preferable.

\section{Numerical tests} \label{sec:tests}
In this section, we test the performance of some of our optimal BDFL and L($\kappa$)-stable LMMs. 
For each method,
we use the diagonally implicit $2^\text{nd}$-order and L-stable one-step TR-BDF2 method \cite[Section 8.5]{leveque2007finite} to generate the $k-1$ starting values. 
For problems with no known analytical solution,
we compute errors with respect to a reference solution generated by Python's 

{\centering \verb|scipy.integrate.solve_ivp|\par}
\noindent using \verb|Radau| with \verb|rtol=1e-10|, \verb|atol=1e-12|.


\subsection{Prothero--Robinson problem} \label{subsec:protherorobinson}
We begin with the classical Prothero--Robinson
test problem 
\cite[Example 8.1]{leveque2007finite}
\begin{align} \label{ode_example1}
    y'(t)=\lambda(y(t)-\cos(t))-\sin(t)
\end{align}
with analytical solution
\begin{align*} 
    y(t)=(y(t_0)-\cos(t_0))e^{\lambda(t-t_0)}+\cos(t).
\end{align*}
We take $\lambda:=-10^6$, $t_0:=0$ and $y(t_0):=1.5$. 
The exact solution
\begin{align} \label{ode_final_solution1}
    y(t)=\frac{1}{2}{e^{-10^6t}}+\cos(t).
\end{align}
For a time step $h=0.1$, the exact solution decays
almost completely to the slow manifold in a single
time step. 
Thus, damping of very fast modes is essential for this problem.

We first test the L-stable BDFL methods.
Errors are plotted in Figure \ref{fig:ode1_err}, where
we see that after the initial transient the error for
each method is roughly proportional to its error constant (see Table \ref{tab:bdfl_345}), with the optimized methods outperforming BDF2.  It is interesting to note that each method shows a much larger error at two instants: step 1 and step $k$.  These are the two points at which the numerical method makes use of the initial data, and the error is orders of magnitude larger at those points simply because that value is far from the slow manifold.

\begin{figure}
    \centering
    \includegraphics[width=0.75\linewidth]{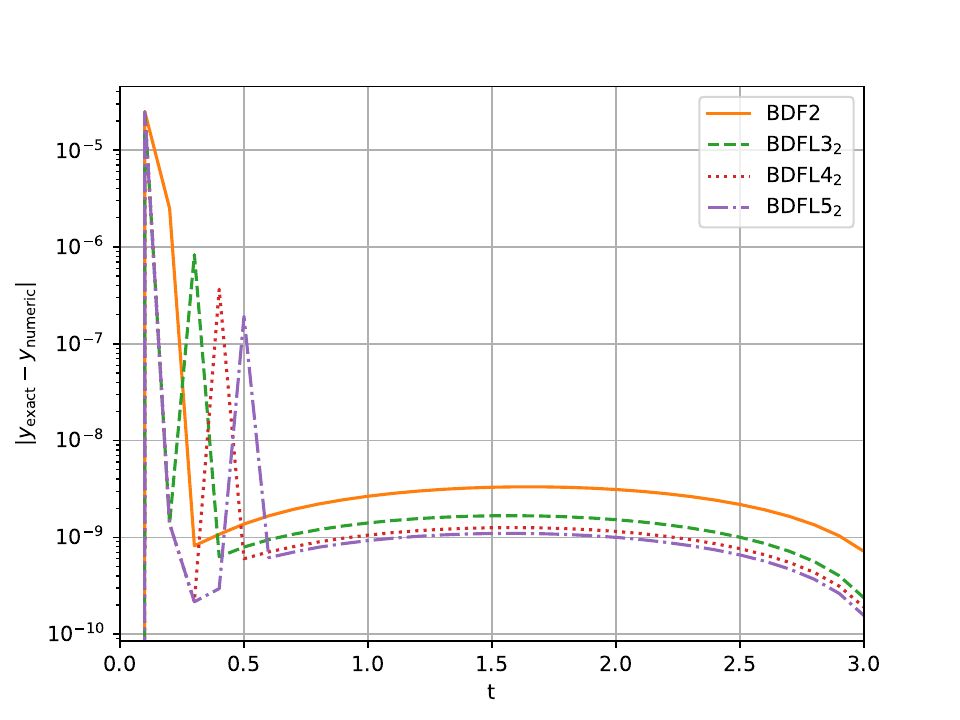}
    \caption{Absolute value
    of the difference between the exact Prothero--Robinson solution \eqref{ode_final_solution1} and 
    the numerical solutions obtained by various LMMs. 
     Note that the second spike in error
    for each LMM corresponds to the transition from the one-step TR-BDF2 method
    to the BDF2 and optimal BDFL methods (Table \ref{tab:bdfl_345}). 
    }
    \label{fig:ode1_err}
\end{figure}

Next we test our optimized L$(\kappa)$-stable methods.
In Figure \ref{fig:lk_perf}, we see very clearly the effect of the asymptotic damping factor; methods with smaller values of $\kappa$ are able to more rapidly damp the initial oscillations.  For the trapezoidal rule with $\kappa=1$, the damping is insignificant on the time scale displayed here.

Errors are plotted in Figure \ref{fig:lk_err}, where for each method we see two regimes: a rapidly-decaying error phase as the oscillations are damped, and then almost-periodic variation of the error once the solution is near the slow manifold.  Interestingly, even the methods with intermediate values of $\kappa$ lead to smaller errors than that of BDF2 once the oscillations are sufficiently damped.

\begin{figure}
    \centering
    \includegraphics[width=0.95\linewidth]{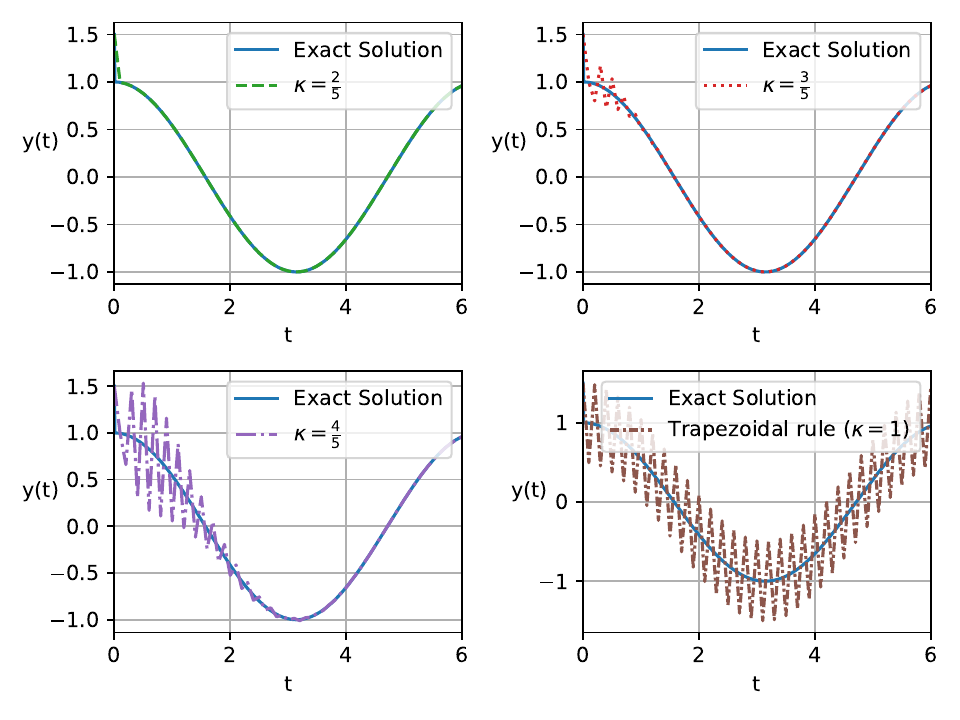}
    \caption{Numerical solutions to the Prothero--Robinson test problem \eqref{ode_example1} using the L$(\kappa)$-stable methods in Table \ref{tab:lk_1234} (the figure with $\kappa=1/5$ is essentially identical to the one with $\kappa=2/5$ hence omitted). Despite having a lower $|C|$, a method with a larger $\kappa$ damps the initial deviation less efficiently, as expected from Definition \ref{Lkappadef}. For $\kappa=4/5$, the deviation even grows
temporarily after the first step; this is the transient overshoot predicted
in Section~\ref{sec:spectral}, caused by the double root $-\kappa$ of
$\sigma$.
    }
    \label{fig:lk_perf}
\end{figure}
\begin{figure}
    \centering
    \includegraphics[width=0.75\linewidth]{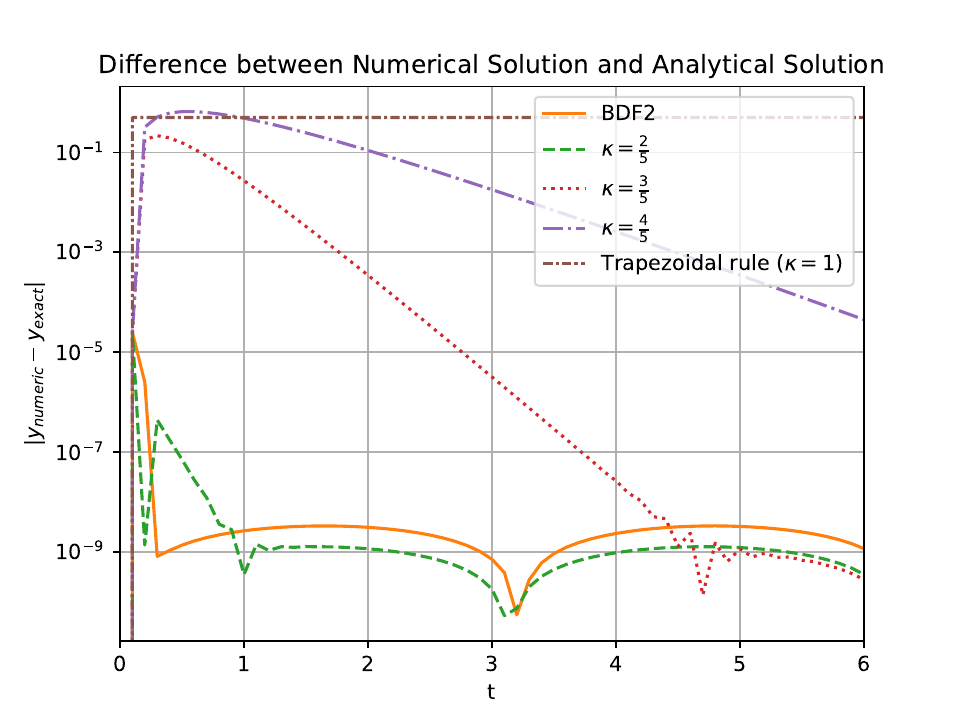}
    \caption{Error from the Prothero--Robinson solution \eqref{ode_final_solution1}. 
    It is evident here that once the initial oscillations of the LMM are sufficiently damped, its performance improves beyond that of BDF2.}
    \label{fig:lk_err}
\end{figure}

\subsection{Van der Pol oscillator} \label{subsec:vanderpolfixed}
Next we consider the nonlinear van der Pol equation \cite[Section IV.1]{wanner1996solving},
written as a first-order system:
\[
\begin{alignedat}{2}
    &y_1'=y_2 \qquad && y_1\left(0\right)=2\\
    &y_2'=\mu\left(1-y_1^2\right)y_2-y_1 \qquad && y_2\left(0\right)=0
\end{alignedat}
\]
with $\mu=500$.  We solve up to time $t=2000$
with each of the optimal BDFL methods (Table \ref{tab:bdfl_345}) with $h=10^{-3}$.  
A reference solution is obtained by the \verb|Radau| method via \verb|scipy.integrate.solve_ivp|.
Solutions and phase errors are shown in Figure \ref{fig:vanderpolsolns}.
It is evident that, with this step size, all methods dramatically lose accuracy after less than one period.

In order to overcome this, one would need to take an extremely
small step size.  With our proposed methods, this would make
the solution very costly.  A more efficient approach is the use of adaptive step size methods, which we develop in the next section.

\begin{figure}
    \centering
    \includegraphics[width=0.75\linewidth]{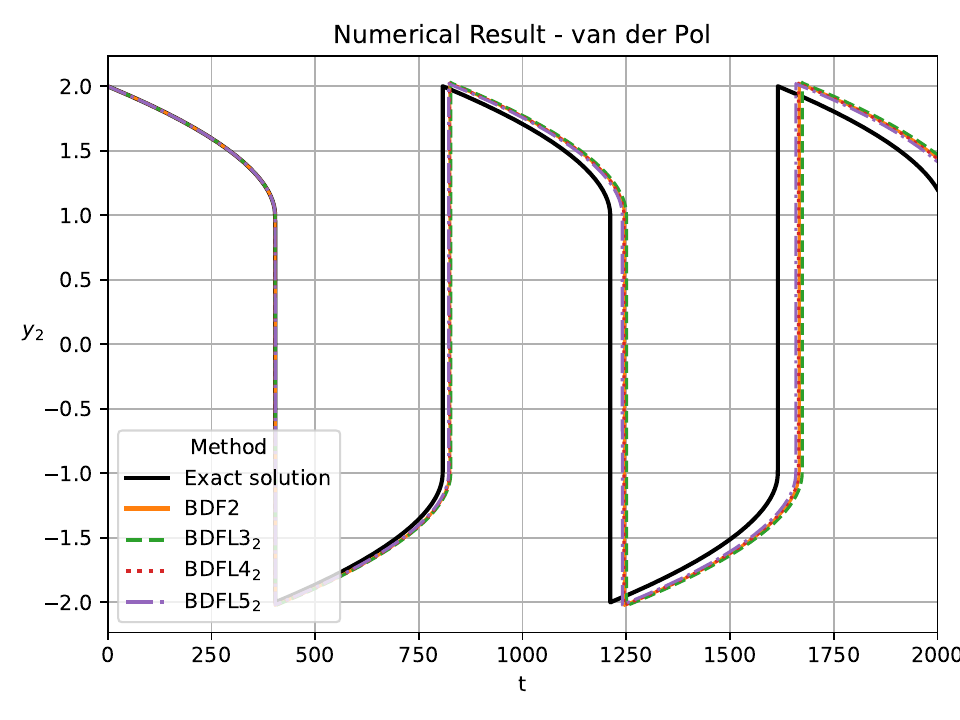}
    \caption{Numerical solution of the van der Pol oscillator with fixed step size methods}
    \label{fig:vanderpolsolns}
\end{figure}

\subsection{Brusselator problem}\label{sec:Brusselator}
Having evaluated the methods on stiff ODEs, we now test them on a stiff PDE: the moderately stiff Brusselator problem \cite[Section 6]{arevalo2021local}.
It involves a one-dimensional reaction-diffusion system

\begin{equation} \label{brusselator}
    \begin{aligned}
    &u_t=A+u^2v-(B+1)u+a\,\,u_{xx} \\
    &v_t=-u^2v+Bu+a\,\,v_{xx},
\end{aligned}
\end{equation}
where we have $A = 1$, $B = 3$, $a = 0.02$, with boundary conditions $u(t, 0) = u(t, 1) = 1$ and $v(t, 0) = v(t, 1) = 3$, as well as initial conditions $v(0, x) = 3$ and $u(0, x) = 1 + 2 \exp(-100(x - 0.4)^2)$. Note that these are the same conditions that appear in \cite[Section 6]{arevalo2021local}. We also use the approximation
\[
    u_{xx}\approx\frac{u_{j-1}-2u_j+u_{j+1}}{\Delta x^2}
    \]

for spatial discretizations, where $\Delta x=1/(N+1)$ 
with $N=60$. The numerical solutions and their errors are displayed in Figures \ref{fig:brussperf}, \ref{fig:brusserrbdfl}, and \ref{fig:brusserrlk}. Errors are measured in the discrete $L^2$ norm relative to a reference solution computed by the {\texttt{Radau}} method \cite[Section 6]{arevalo2021local}, and defined as
\begin{align*}
    \|e\|_{L^2} := \sqrt{\Delta x \sum_j e_j^2}.
\end{align*}
To assess the effect of system size and stiffness, we have repeated
this test with $N=240$ and $N=960$ interior grid points (systems of
dimension $480$ and $1920$, with $h|\lambda|_{\max}\approx 46$ and $740$ for
the diffusion part, compared with about $3$ for $N=60$). The errors of all
methods change by less than $0.5\%$, so the accuracy and the ranking of the
methods are unaffected by the size and stiffness of the system.

Table~\ref{tab:trbdf2} compares the errors of our methods with those
of TR-BDF2 at equal cost, for this problem and for the Prothero--Robinson
problem of Section~\ref{subsec:protherorobinson}. Since each TR-BDF2 step requires two
implicit solves, TR-BDF2 is applied with step size $2h$; we consider both
$\gamma=1/2$ and the standard choice $\gamma=2-\sqrt{2}$. At equal cost,
BDFL$4_2$ and BDFL$5_2$ are more accurate than TR-BDF2 on both problems, and
BDFL$3_2$ is at least as accurate. The L$(\kappa)$-stable methods are more
accurate than TR-BDF2 on the Brusselator problem, and also on the
Prothero--Robinson problem for $\kappa\le 2/5$; for $\kappa\in\{3/5,4/5\}$,
the error there is dominated by the slowly damped initial transient (cf.\
Figure~\ref{fig:lk_perf}).

\begin{table}[tbp]
  \centering
  \renewcommand{\arraystretch}{1.2}
  \caption{Maximum error for $t\ge 1$ at equal cost (equal number of
    implicit solves). Prothero--Robinson problem: $\lambda=-10^6$, $h=0.1$,
    $t\in[0,6]$; Brusselator problem: $N=60$, $h=0.01$, $t\in[0,20]$,
    $L^2$ error in $u$. TR-BDF2 uses step size $2h$.}
  \label{tab:trbdf2}
  \begin{tabular}{|l|c|c|}
    \hline
    \textbf{Method} & \textbf{Prothero--Robinson} & \textbf{Brusselator} \\
    \hline
    BDF2                                  & $3.33\cdot10^{-9}$ & $6.21\cdot10^{-4}$ \\
    BDFL$3_2$                             & $1.68\cdot10^{-9}$ & $3.02\cdot10^{-4}$ \\
    BDFL$4_2$                             & $1.26\cdot10^{-9}$ & $2.25\cdot10^{-4}$ \\
    BDFL$5_2$                             & $1.10\cdot10^{-9}$ & $1.93\cdot10^{-4}$ \\
    \hline
    L$(\kappa)$-stable, $\kappa=1/5$       & $1.46\cdot10^{-9}$ & $2.66\cdot10^{-4}$ \\
    L$(\kappa)$-stable, $\kappa=2/5$       & $1.43\cdot10^{-9}$ & $2.34\cdot10^{-4}$ \\
    L$(\kappa)$-stable, $\kappa=3/5$       & $2.72\cdot10^{-2}$ & $1.88\cdot10^{-4}$ \\
    L$(\kappa)$-stable, $\kappa=4/5$       & $4.83\cdot10^{-1}$ & $1.65\cdot10^{-4}$ \\
    \hline
    TR-BDF2, $\gamma=1/2$                 & $3.33\cdot10^{-9}$ & $3.22\cdot10^{-4}$ \\
    TR-BDF2, $\gamma=2-\sqrt{2}$          & $2.76\cdot10^{-9}$ & $3.13\cdot10^{-4}$ \\
    \hline
  \end{tabular}
\end{table}

\begin{figure}
    \centering
    \includegraphics[width=0.75\linewidth]{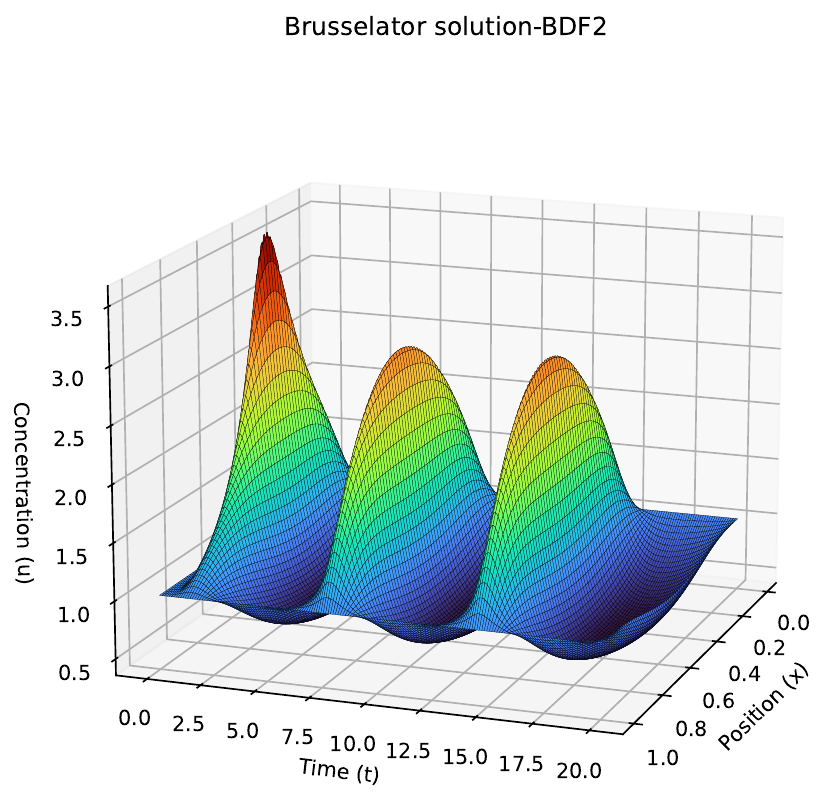}
    \caption{
    Numerical solution $u(x,t)$ of the Brusselator problem computed using BDF2. Solutions obtained with the $\mathrm{BDFL}$ and $\mathrm{L}(\kappa)$-stable families are visually indistinguishable and are thus omitted for brevity.}
    \label{fig:brussperf}
\end{figure}
\begin{figure}
    \centering
    \includegraphics[width=0.85\linewidth]{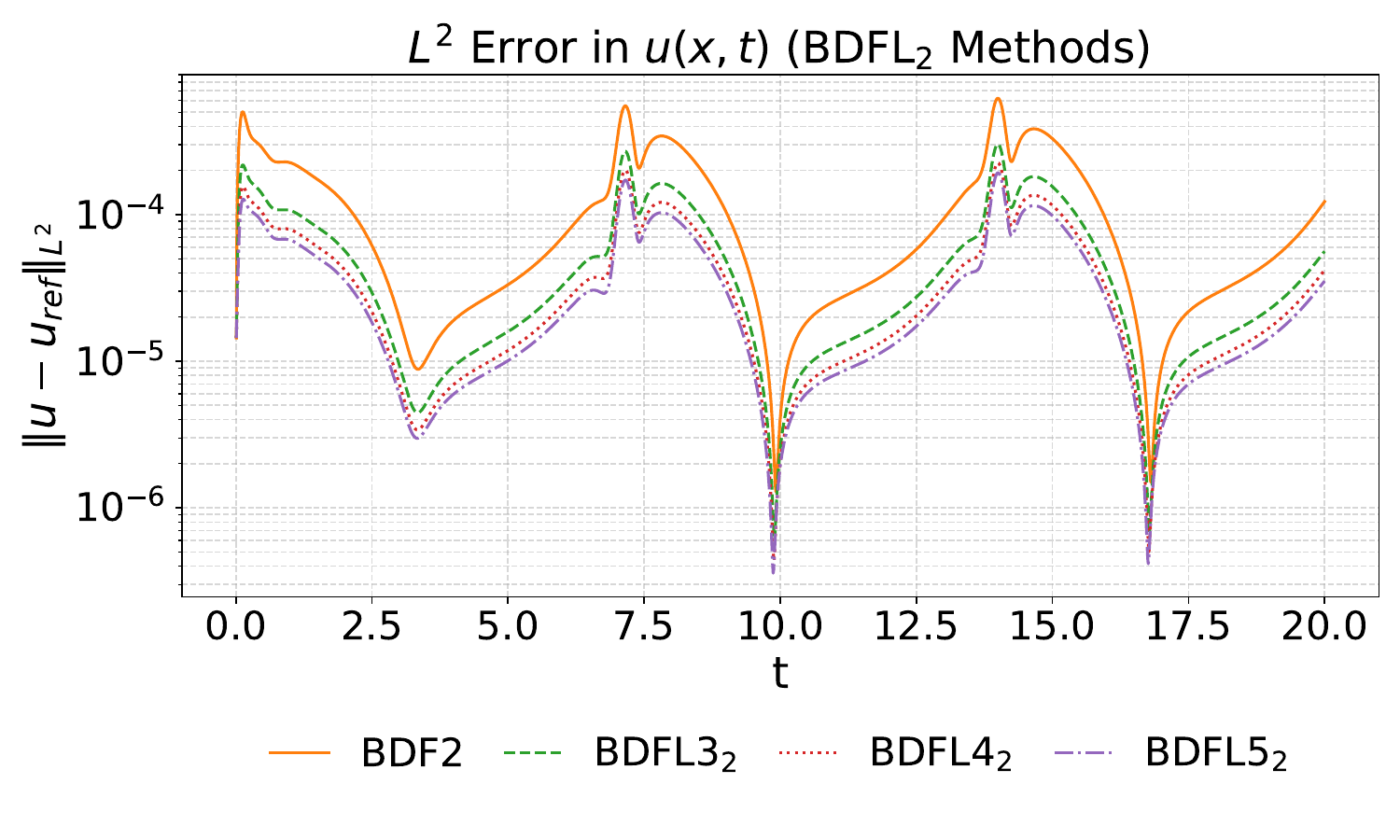}
    \caption{Discrete $L^2$ error in $u(x,t)$ over time for BDF2 and the BDFL methods from Table \ref{tab:bdfl_345} on the Brusselator problem (relative to a Radau reference solution), showing consistently lower errors for higher-step BDFL schemes. 
    }
    \label{fig:brusserrbdfl}
\end{figure}
\begin{figure}
    \centering
    \includegraphics[width=0.85\linewidth]{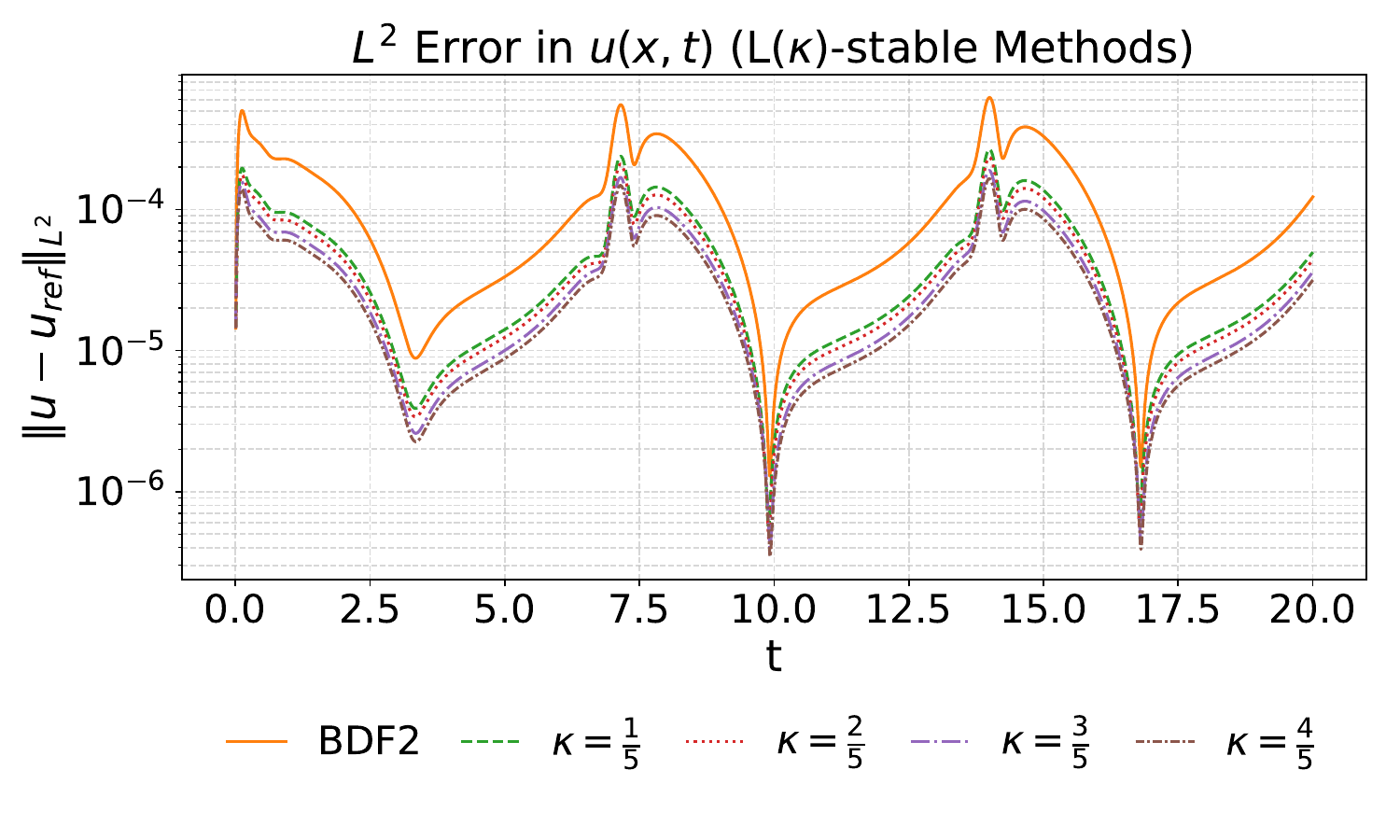}
    \caption{
    Discrete $L^2$ error in $u(x,t)$ over time for BDF2 and the 
    $\mathrm{L}(\kappa)$-stable methods from Table \ref{tab:lk_1234} on the Brusselator problem (relative to a Radau reference solution), showing errors that decrease consistently with smaller error constants $|C|$ as $\kappa$ increases.}
    \label{fig:brusserrlk}
\end{figure}

\section{Variable-step-size formulation} \label{sec:vss}
Efficient application of LMMs to most problems of interest requires step-size adaptation.  In this section, we develop a proof-of-concept 
variable-step-size (VSS) extension for each of the optimal BDFL$k_2$ methods.
The goal here is only to show that a reasonably efficient VSS extension is possible.
In particular, we do not prove zero-stability of the resulting
variable-step methods for arbitrary sequences of step-size ratios; see
Section~\ref{subsec:vss_stable_regions}.
A more sophisticated design could involve
trying to minimize the error in some sense for a range of 
step-size ratios, and employ a more advanced step-size controller, but that is left to future work.

\subsection{Setup for VSS}
Our VSS methods make use of coefficients that depend on the step size; we therefore write the methods in the form
\begin{align} \label{lmmvariablestep0}
    y_{n+k}+\sum^{k-1}_{j=0}\alpha_{j,n+k-1}y_{n+j}=h_{n+k-1}\sum^k_{j=0}\beta_{j,n+k-1}f_{n+j}.
\end{align}
For the variable-step-size analysis it is convenient to index a step by the interval $h_n=t_{n+1}-t_n$ ending at its new solution value. Thus, replacing $n$ by $n+1-k$ in \eqref{lmmvariablestep0}, and indexing the variable coefficients accordingly, we write
\begin{align}\label{lmmvariablestep}
\sum_{j=0}^k \alpha_{j,n}y_{n+1-k+j}
=
h_n\sum_{j=0}^k\beta_{j,n}f_{n+1-k+j},
\qquad \alpha_{k,n}=1.
\end{align}
The local truncation error for \eqref{lmmvariablestep} is given by
\begin{align} \label{ltevariablestep}
    \tau_n & = \frac{1}{h_n} \left(\sum_{j=0}^k \alpha_{j,n}\right) y(t_{n+k-1}) + \sum_{p=1}^\infty h_n^{p-1} \sum_{j=0}^k \left( \frac{c_{j,n}^p}{p!}\alpha_{j,n} - \frac{c_{j,n}^{p-1}}{(p-1)!}\beta_{j,n} \right)y^{(p)}(t_{n+k-1}),
\end{align}
where
$$
c_{j,n}:=\frac{t_{n+1-k+j}-t_{n+1-k}}{h_n} = \frac{1}{h_n} \sum_{i=1}^j h_{n-k+i}.
$$
Of course, \eqref{ltevariablestep} reduces to
\eqref{localtruncationerror} in the case of constant step size. From \eqref{ltevariablestep} we obtain the conditions for second order:
\begin{subequations} \label{consistencyvariable}
\begin{align}
    \sum_{j=0}^k \alpha_{j,n} & = 0 \\
    \sum_{j=0}^k \left(c_{j,n} \alpha_{j,n} - \beta_{j,n}\right) & = 0 \\
    \sum_{j=0}^k \left(\frac{1}{2}c_{j,n}^2 \alpha_{j,n} - c_{j,n} \beta_{j,n}\right) & = 0.
\end{align}
\end{subequations}
In what follows, it is convenient to work with the step size ratios $\omega_n:=\frac{h_n}{h_{n-1}}$. In terms of these, we have
\[
    c_{j,n} = \sum_{i=1}^j \prod_{\ell=n-k+i+1}^{n} \frac{1}{\omega_\ell}.
\]

If we consider a class of methods with only three free coefficients, then \eqref{consistencyvariable} is enough to fully determine the (adaptive) method.  In this manner, one can obtain
for instance a variable-step-size BDF2 method \cite[Section III.5]{hairer1993solving}. A similar exercise can be applied to derive the implicit Adams method ($k=2$) found in the same section and the BDF3 and BDF4  methods found in \cite{decaria2021variable}.

\subsection{Adaptive BDFL methods}\label{sec:adaptiveBDFmethods}
However, for our optimal BDFL methods, the number of free coefficients is greater than the number of consistency conditions.  We initially attempted to choose the adaptive coefficients in a way that would minimize the error constant for some distribution of step size ratios, but this problem proved intractable.  Here we instead eliminate the extra degrees of freedom by setting some of the coefficients to their fixed step size values, and then solve the consistency conditions to determine the remaining coefficients. With this choice, the variable-step methods reduce to the optimal
constant-step BDFL$k_2$ methods when the step size is constant; however,
they are not optimized for varying step sizes (see the discussion at the end
of Section~\ref{subsec:vanderpolvariable}).

For instance, to find an adaptive version of BDFL$3_2$, we use the structural conditions $\beta_{0,n}=\beta_{1,n}=\beta_{2,n}=0,\alpha_{3,n}=1$, and additionally set $\alpha_{2,n}=-\frac{3}{2}$.  We then solve \eqref{consistencyvariable} for $\alpha_{0,n}, \alpha_{1,n}$, and $\beta_{3,n}$.  The resulting method is
\begin{align*}
    \alpha_{0,n} &= \frac{\omega _{n-1}^2 \left(2 \omega _n^2-2 \omega _n-1\right)}{4 \omega _{n-1} \left(\omega _n+1\right)+2} \\
    \alpha_{1,n}&=\frac{\left(-2 \omega _n^2+2
   \omega _n+1\right) \omega _{n-1}^2+2 \left(\omega _n+1\right) \omega _{n-1}+1}{4 \omega _{n-1} \left(\omega
   _n+1\right)+2} \\
    \alpha_{2,n}&=-\frac{3}{2} \\
    \alpha_{3,n}&=1 \\
    \beta_{3,n}&=\frac{2 \omega_n+\omega_{n-1} \left(2 \omega _n^2+4 \omega _n-1\right)-1}{2 \omega _n \left(2 \omega _{n-1} \left(\omega_n+1\right)+1\right)}.
\end{align*}

To determine adaptive versions of
BDFL$4_2$ and BDFL$5_2$, we additionally fixed the values of $\alpha_{2,n}, \alpha_{3,n}, \alpha_{4,n}$, and (for BDFL$5_2$) $\alpha_{5,n}$.
The coefficients of the resulting variable-step-size LMMs are given in Appendix \ref{varstepbdfl45coeffs}.


\subsection{Stability} \label{subsec:vss_stable_regions}
Stability of adaptive methods is much more complicated than for fixed step size methods \cite{geart1974effect,grigorieff1983stability,soderlind2018zero,hadjimichael2016strong}; see \cite[Section III.5]{hairer1993solving} for an overview and \cite{li2022stability} for an analysis of BDF methods.
Here, rather than studying products of varying-coefficient companion matrices or assuming smooth step-size variation,
we undertake a frozen-coefficient stability analysis: for fixed values of the recent step-size ratios, we determine whether the corresponding characteristic polynomial satisfies the root condition. This provides a practical indication of admissible step-size variations, although it does not characterize zero-stability for arbitrary sequences of varying step ratios.
More precisely, the frozen-coefficient root condition is necessary for
stability when the step-size ratios are kept constant, but it does not
guarantee zero-stability when the ratios vary from step to step, even if
each ratio lies in the admissible set
\cite{geart1974effect,grigorieff1983stability,soderlind2018zero}. A rigorous
variable-step stability analysis, as carried out for BDF2 and BDF3 in
\cite{li2022stability}, is beyond the scope of this work.


For BDFL$3_2$, we can determine the exact set of step size ratios
$(\omega_n, \omega_{n-1})$ for which this frozen-coefficient root condition holds:
\begin{align*}
&
\Biggl(
0 < \omega_n < \frac{1+\sqrt{3}}{2}
\;\land\;
\Biggl(
    \frac{1}{8}
    \left(
        \frac{\omega_n+1}{-2\omega_n^2+2\omega_n+1}
        +
        \sqrt{3}
        \sqrt{
            \frac{-5\omega_n^2+6\omega_n+3}
            {\left(-2\omega_n^2+2\omega_n+1\right)^2}
        }
    \right)
    \le
    \omega_{n-1}
\\[1ex]
&\qquad\qquad\qquad
<
\frac{2(\omega_n+1)}
{-2\omega_n^2+2\omega_n+1}
+
\sqrt{
    \frac{12\omega_n+6}
    {\left(-2\omega_n^2+2\omega_n+1\right)^2}
} \notag
\\[1ex]
&\qquad\qquad
\lor
\qquad
0<\omega_{n-1}
<
\frac{1}{8}
\left(
    \frac{\omega_n+1}
    {-2\omega_n^2+2\omega_n+1}
    +
    \sqrt{3}
    \sqrt{
        \frac{-5\omega_n^2+6\omega_n+3}
        {\left(-2\omega_n^2+2\omega_n+1\right)^2}
    }
\right)
\Biggr)
\Biggr)\notag
\\[2ex]
&\lor
\left(
2\omega_n=1+\sqrt{3}
\;\land\;
\omega_{n-1}>0
\right) \notag
\\[2ex]
&\lor
\left(
2\omega_n>1+\sqrt{3}
\;\land\;
0<\omega_{n-1}
<
\frac{\omega_n+1}
{2\omega_n^2-2\omega_n-1}
+
\sqrt{3}
\sqrt{
\frac{\omega_n^2}
{\left(-2\omega_n^2+2\omega_n+1\right)^2}
}
\right).\notag
\end{align*}
This set is plotted in Figure \ref{fig:bdfl3varstable}.
\begin{figure}[H]
    \centering
    \includegraphics[width=0.5\linewidth]{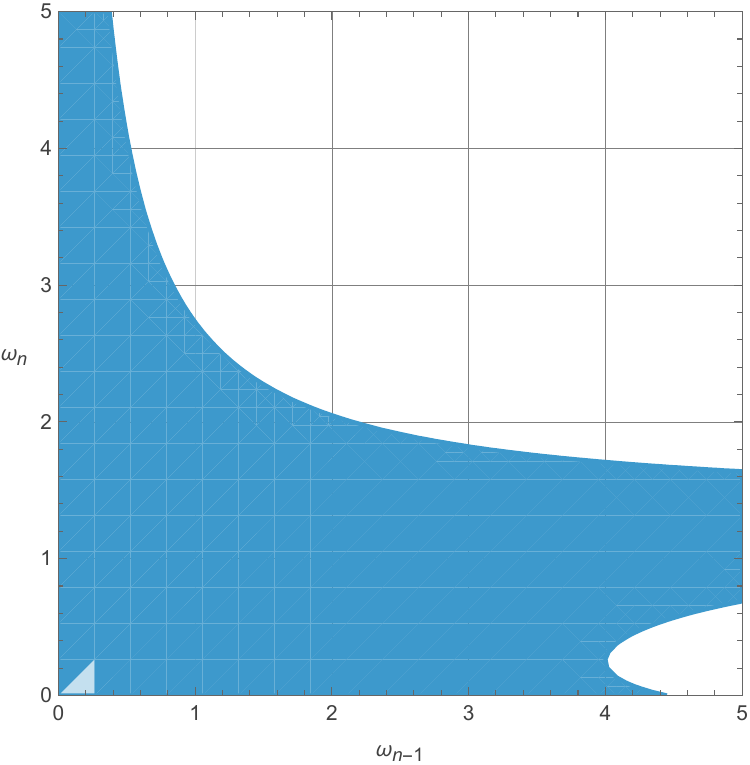}
    \caption{Set of step-size ratios satisfying the frozen-coefficient root condition for variable-step-size BDFL$3_2$.} 
    \label{fig:bdfl3varstable}
\end{figure}

For BDFL$4_2$ and BDFL$5_2$, the (higher-dimensional) stable sets are too complex to characterize or plot, but \cite[Theorem III.5.5]{hairer1993solving} guarantees that they include some neighborhood of the point given by $\omega_j=1$ for all $j$.

\subsection{Error estimation and step-size control} \label{subsec:adaptabless}

For adaptive step-size selection, we estimate the local error using a cubic
Hermite extrapolation. Given the two most recent accepted solution values
$y_{n-1},y_n$ and their corresponding derivatives $f_{n-1},f_n$, we construct
the cubic polynomial $p(t)$ satisfying
\[
p(t_{n-1})=y_{n-1}, \qquad p'(t_{n-1})=f_{n-1},
\]
\[
p(t_n)=y_n, \qquad p'(t_n)=f_n.
\]
We then extrapolate this polynomial to the next time level and define
\[
z_{n+1}=p(t_{n+1}).
\]
Since cubic Hermite extrapolation has error of order $\mathcal{O}(h^4)$,
while the BDFL methods considered here have local solution error of order
$\mathcal{O}(h^3)$, the difference
\[
E_n:=\|y_{n+1}-z_{n+1}\|
\]
provides an asymptotically correct estimate of the leading local error of the
second-order LMM. This estimator requires no additional evaluations of the ODE
right-hand side, since $f_{n-1}$ and $f_n$ are already available from previous
accepted steps. We use it in the standard controller
\[
h_{n+1}
=
\eta h_n
\left(\frac{\mathrm{TOL}}{E_n}\right)^{1/3},
\]
with a safety factor $0<\eta<1$ and upper and lower bounds on the permitted
step-size ratio.
In our implementation, we use $\eta=0.9$ and restrict the step-size
ratios to $0.1\le\omega_n\le 1.5$. For each of the methods, all
frozen-coefficient characteristic polynomials with ratios in this range
satisfy the root condition (verified on a grid of ratio values), with
spurious roots of modulus at most $0.88$.

Since this is only a proof-of-concept implementation, we have focused on a relatively simple estimator and controller.
For state-of-the-art analysis and design of error estimators and controllers for LMMs, see for instance \cite{arevalo2021local}.


\subsection{Variable-step-size methods: Robertson's problem} \label{subsec:robertson}

We test these methods on two problems.
First, we consider Robertson's problem \cite{robertson1966solution},\cite[Section IV.1]{wanner1996solving}:
\begin{equation} \label{robertson}
\begin{alignedat}{2}
    y'_1&=-0.04 y_1+10^4 y_2 y_3 \qquad && y_1(0)=1\\
    y'_2&=0.04 y_1-10^4 y_2 y_3-(3\times10^7)y_2^2 \qquad \qquad \qquad&& y_2(0)=0\\
    y'_3&=(3\times10^7)y_2^2 \qquad && y_3(0)=0.
\end{alignedat}
\end{equation}
The rate of change of the solution for this problem is known to vary dramatically, with a rapid rate of change at early times and much slower rate at later times.  Therefore, the use of a variable step size is essential to its efficient solution. Due to the known logarithmic structure of the solution, instead of choosing the step size adaptively based on the error, we simply choose a constant ratio between successive steps. This was done using numpy's \verb|logspace| function, where we arbitrarily chose 2000 logarithmically spaced points from $10^{-6} \text{ to } 10^8$ seconds. This results in a step-size ratio of about 1.0162569.

%

A reference solution is obtained by the \verb|Radau| method via \verb|scipy.integrate.solve_ivp|.
Errors are plotted in Figure \ref{fig:robertson_err}, where we see again that the error for each method is roughly proportional to its error constant.

\begin{figure}[H]
    \centering
    \includegraphics[width=0.7\linewidth]{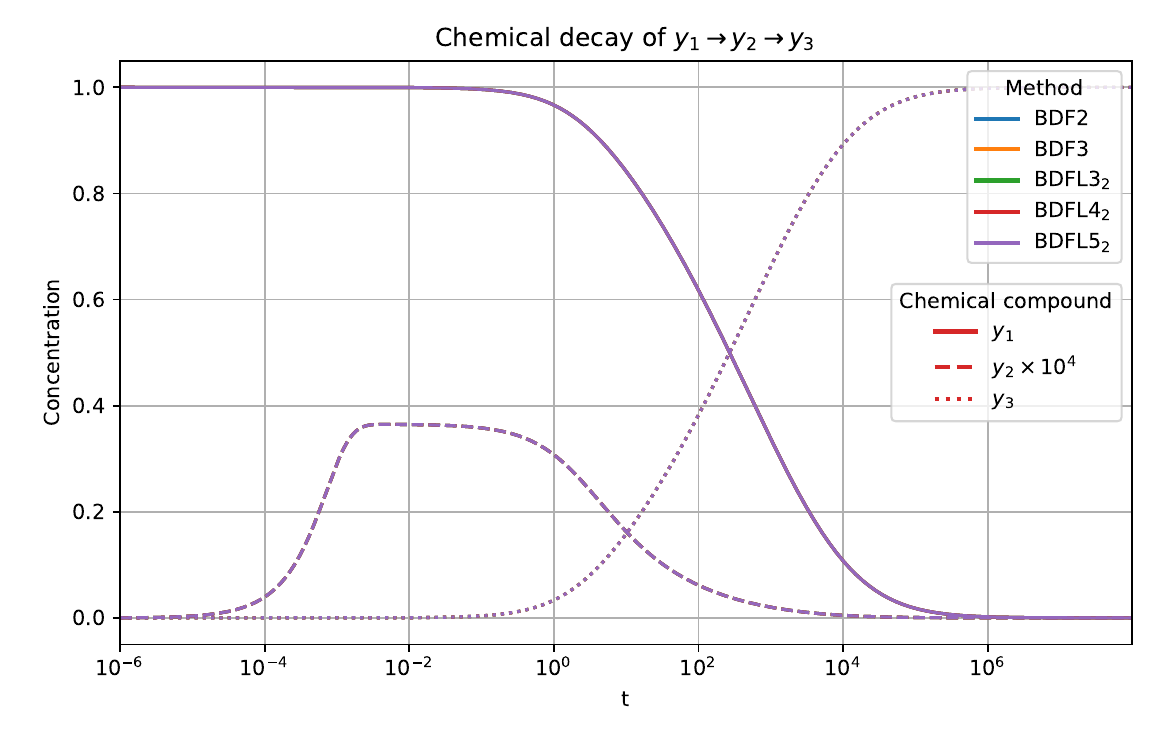}
    \caption{Solution to Robertson's problem \eqref{robertson} generated by BDF2, BDF3 and the optimal BDFL methods (Table \ref{tab:bdfl_345}).  At this scale and with the step sequence used here, all solutions are essentially indistinguishable  from the exact solution.}
    \label{fig:robertson_perf}
\end{figure}
\begin{figure}[H]
         \centering
     \begin{subfigure}[b]{0.75\textwidth}
         \centering
         \includegraphics[width=\linewidth]{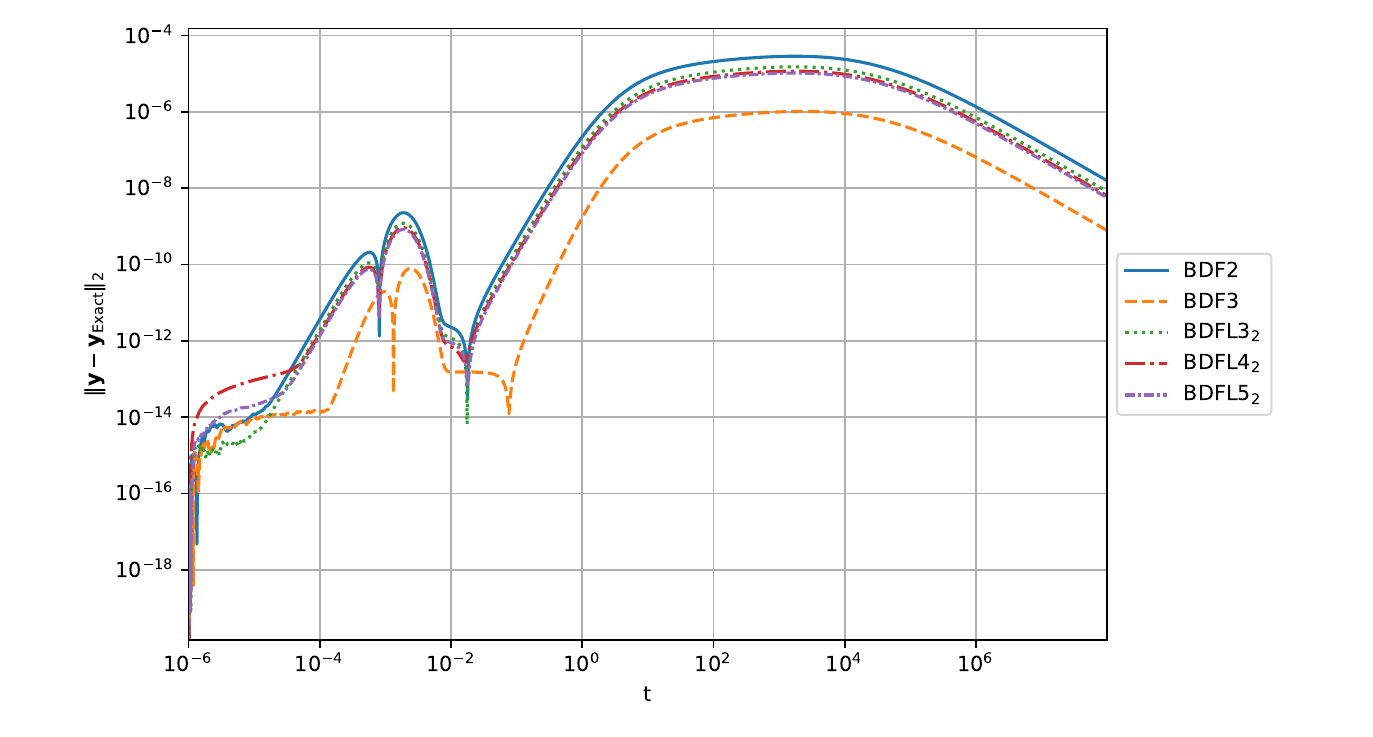}
     \end{subfigure}\\
     \begin{subfigure}[b]{0.75\textwidth}
         \centering
         \includegraphics[width=\linewidth]{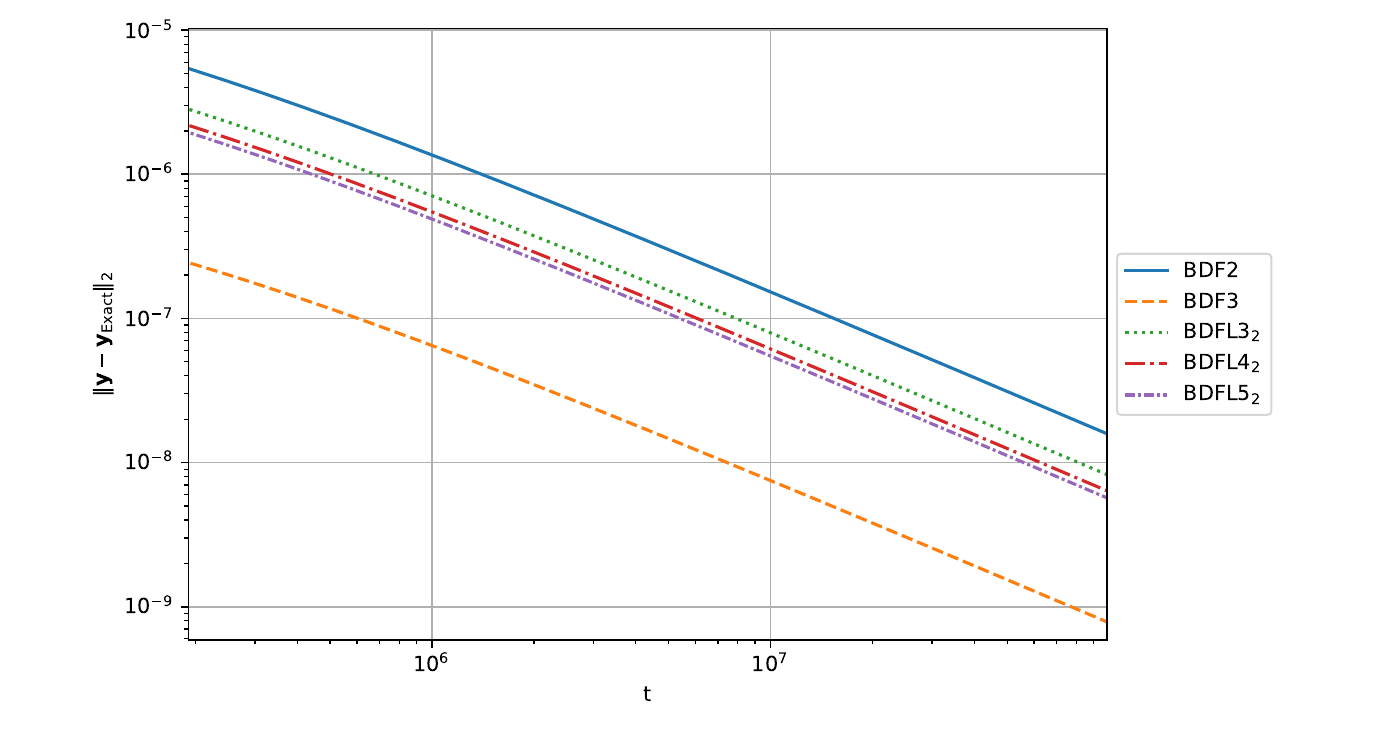}
     \end{subfigure}
    \caption{Norm of the difference between the BDF and BDFL$k_2$ methods, and the exact solution generated by \texttt{scipy.integrate.solve\_ivp}. The second figure includes a zoom-in of BDF2, BDF3 and the optimal BDFL methods. Note how we see a similar performance of the optimal BDFL methods as from Figure \ref{fig:ode1_err}.}
    \label{fig:robertson_err}
\end{figure}

\subsection{Van der Pol oscillator revisited} \label{subsec:vanderpolvariable}
Next, we consider again the van der Pol equation,
this time with adaptive step-size control.
We again take $\mu=500$ and solve up to time $t=2000$,
applying our variable-step-size methods and using the step-size-control strategy described
in Section \ref{subsec:adaptabless}, with a tolerance of $10^{-6}$.  The resulting solutions, shown in the right plot of Figure \ref{fig:vanderpolsolns-vss}, remain highly accurate after multiple periods.  Each adaptive method takes about $3.38\times 10^5$ steps to complete the integration, in contrast to the $2 \times 10^6$ steps taken in the 
fixed-step-size case.

\begin{figure}
    \centering
    \includegraphics[width=0.45\linewidth]{Python_files/van_der_pol_performance_mu500.pdf}
    \hfill
    \includegraphics[width=0.45\linewidth]{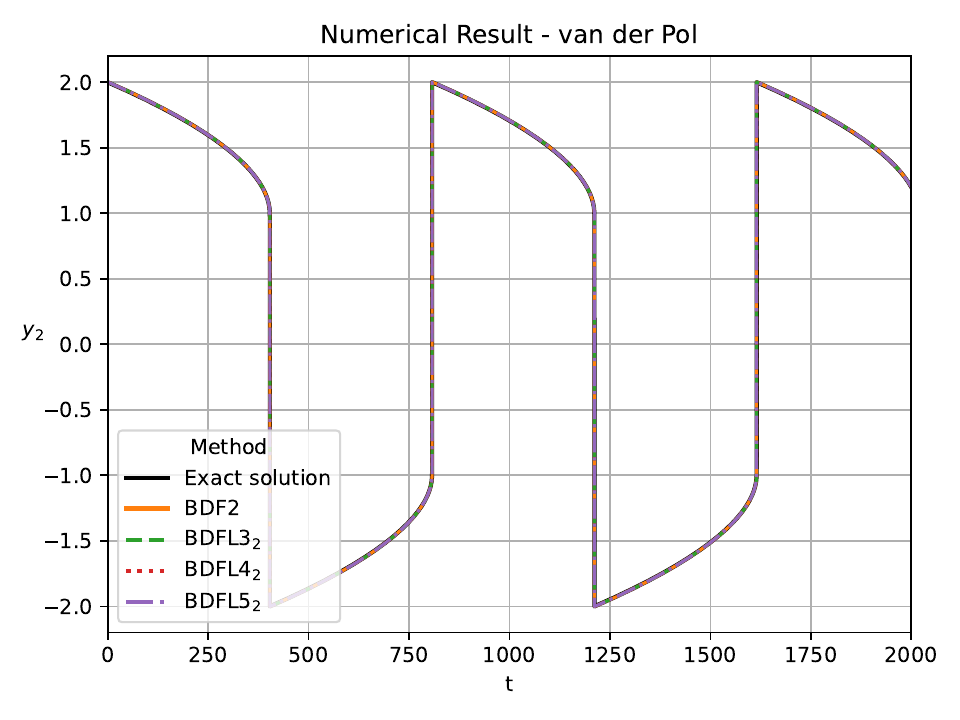} 
    \caption{Numerical solution of the van der Pol oscillator.  Left: fixed-step-size methods.  Right: variable-step-size methods.}
    \label{fig:vanderpolsolns-vss}
\end{figure}

Finally, we conduct a work-precision study to evaluate the efficiency of these methods relative to BDF2, see Figure \ref{fig:workprec}.  We solve the problem up to $t=2000$, with prescribed tolerances in the range $[10^{-6},10^{-3}]$ and compute the 
$L^2$
norm of the error over the full time interval.  Compared to BDF2, the proposed methods provide noticeably improved accuracy for the same cost over most of this range.  For the BDFL$4_2$ and BDFL$5_2$ methods, the fraction of rejected steps is high (as much as 30\%) even at tight tolerances.  We attribute this to the ad hoc manner in which the VSS coefficients were designed.  It would be valuable to develop a more sophisticated choice of coefficients that ensures small error constants over a wider range of step-size ratios, and we expect that this would yield an even greater efficiency gain for these methods, particularly at tight tolerances. 
For comparison, we have also solved this problem with the production
codes \texttt{BDF} and \texttt{Radau} of \texttt{scipy.integrate.solve\_ivp}.
As expected, these higher-order codes are considerably more efficient. When
\texttt{BDF} is restricted to orders $1$ and $2$, it uses the formula NDF2,
i.e., BDFL$3_2$, with a quasi-constant step size, a simplified Newton
iteration and order switching; it requires about half as many steps and an
order of magnitude fewer right-hand side evaluations than our
proof-of-concept implementation. This indicates that the efficiency of our
variable-step implementations could be improved substantially.
Possible strategies include choosing the free coefficients so that the
variable-step error constant is stationary with respect to the step-size
ratios at constant step size, which is a tractable local version of the
optimization problem mentioned in Section~\ref{sec:adaptiveBDFmethods};
choosing them to enlarge the set of admissible step-size ratios of
Section~\ref{subsec:vss_stable_regions}; using a quasi-constant step size
with interpolation, or a fixed-leading-coefficient or Nordsieck-type
formulation, as is common in variable-step BDF codes; a simplified Newton
iteration with Jacobian reuse; and replacing the elementary controller by a
smoother one, such as a PI controller or a digital filter, which typically
yields smoother step-size sequences and fewer rejected steps
\cite{arevalo2021local}. We leave these to future work.

\begin{figure}
    \centering
    \includegraphics[width=0.45\linewidth]{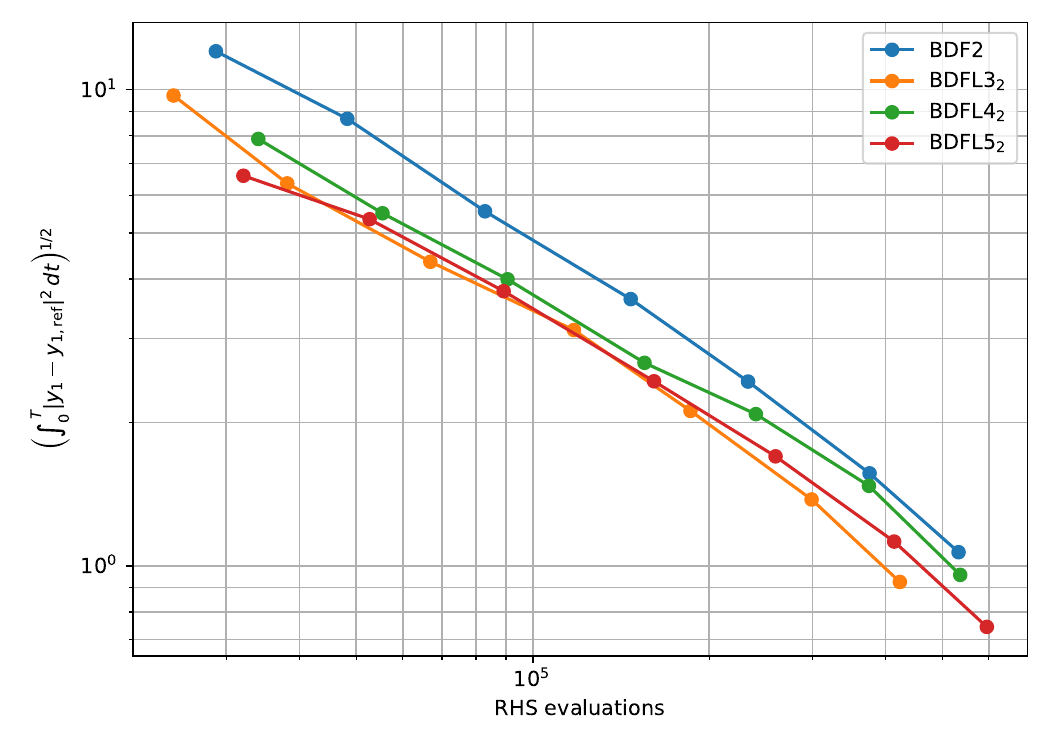}
    \hfill
    \includegraphics[width=0.45\linewidth]{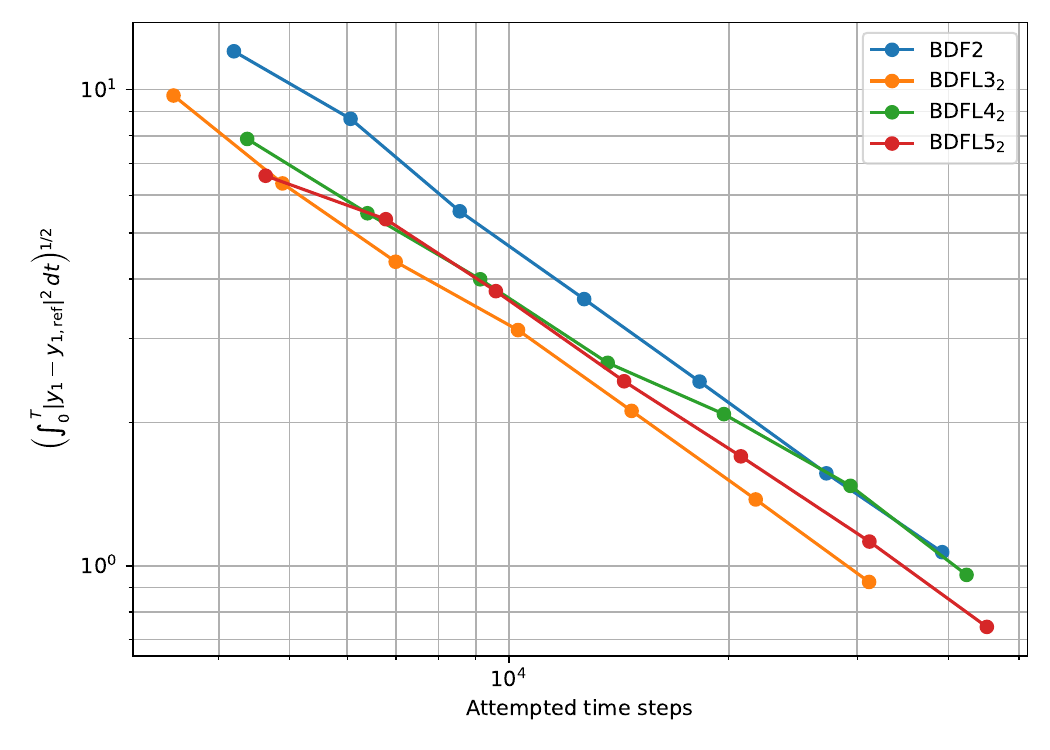} 
    \caption{Work-precision study for VSS methods applied to the van der Pol problem.  Left: number of RHS evaluations. Right: number of attempted steps.}
    \label{fig:workprec}
\end{figure}

\appendix
\section{Variable-step-size coefficients for the optimal BDFL$4_2$ and BDFL$5_2$ methods} \label{varstepbdfl45coeffs}

Below, we present the coefficients of the adaptive versions of
our BDFL$4_2$ and BDFL$5_2$ methods, see Section \ref{sec:adaptiveBDFmethods}.

$\bullet$ For the BDFL$4_2$ method, using the abbreviations
\begin{align}
    D_4 &= 82\,\omega_{n-2}
            \bigl(\omega_{n-1}(\omega_n+1)+1\bigr) + 41
     \notag\\[4pt]
    A_4 &= \bigl(\omega_n(41\omega_n - 44\sqrt{2}+10)
            - 22\sqrt{2}+5\bigr)\omega_{n-1}^2 \notag\\
        &\quad + 2(7+2\sqrt{2})(\omega_n+1)\omega_{n-1}
            + 2\sqrt{2}+7
     \notag\\[4pt]
    B_4 &= \bigl((-41\omega_n+44\sqrt{2}-10)\omega_n
            + 22\sqrt{2}-5\bigr)\omega_{n-1}^2 \notag\\
        &\quad - 2(7+2\sqrt{2})(\omega_{n-2}+1)(\omega_n+1)\omega_{n-1}
            - (7+2\sqrt{2})(\omega_{n-2}+2)
     \notag\\[4pt]
    N_{\beta_4} &=
        \omega_{n-2}\bigl(41\omega_n(\omega_n+2)-22\sqrt{2}+5\bigr)
            \omega_{n-1}^2 \notag\\
        &\quad - (2\omega_{n-2}+1)(-41\omega_n+22\sqrt{2}-5)\,\omega_{n-1}
            + (7+2\sqrt{2})(\omega_{n-2}+1)
     \notag\\[4pt]
    D_{\beta_4} &= 41\,\omega_{n-1}\,\omega_n
        \bigl(2\omega_{n-2}
            \bigl(\omega_{n-1}(\omega_n+1)+1\bigr)+1\bigr),
     \notag
\end{align}
the method coefficients are given by
\begin{align*}
    \alpha_{0,n} &= \frac{\omega_{n-2}^2\, A_4}{D_4}  \\[4pt]
    \alpha_{1,n} &= \frac{\omega_{n-2}(\omega_{n-2}\, B_4) - 2\sqrt{2}-7}{D_4} \\[4pt]
    \alpha_{2,n} &= \tfrac{2}{41}(1+12\sqrt{2}) \\[4pt]
    \alpha_{3,n} &= -\tfrac{2}{41}(18+11\sqrt{2}) \\[4pt]
    \alpha_{4,n} &= 1 \\[4pt]
    \beta_{4,n}  &= \frac{N_{\beta_4}}{D_{\beta_4}}.
\end{align*}

$\bullet$ For the BDFL$5_2$ method, using the abbreviations
\begin{align}
    D_5 &= 88\Bigl(2\omega_{n-3}
            \bigl(\omega_{n-2}
                \bigl(\omega_{n-1}(\omega_n+1)+1\bigr)+1\bigr)+1\Bigr)
     \notag\\[4pt]
    P_5 &= \bigl(2(-22\omega_n+14\sqrt{5}+1)\omega_n
            +14\sqrt{5}+1\bigr)\omega_{n-1}^2 \notag\\
        &\quad + (2-16\sqrt{5})(\omega_n+1)\omega_{n-1}
            - 8\sqrt{5}+1
     \notag\\[4pt]
    Q_5 &= \omega_{n-1}(\omega_n+1)+1
     \notag\\[4pt]
    N_{\beta_5} &=
        2\omega_{n-3}\,\omega_{n-2}^2
            (-44\omega_n(\omega_n+2)+14\sqrt{5}+1)\,\omega_{n-1}^2
            \notag\\
        &\quad + 2\omega_{n-2}(2\omega_{n-3}(\omega_{n-2}+1)+1)
            (-44\omega_n+14\sqrt{5}+1)\,\omega_{n-1} \notag\\
        &\quad + (2-16\sqrt{5})\,\omega_{n-2} \notag\\
        &\quad + \omega_{n-3}\bigl(
            (2-16\sqrt{5})\,\omega_{n-2}(\omega_{n-2}+2)
            -7\sqrt{5}+27\bigr)
            -7\sqrt{5}+27
     \notag\\[4pt]
    D_{\beta_5} &= 88\,\omega_{n-2}\,\omega_{n-1}\,\omega_n
        \Bigl(2\omega_{n-3}
            \bigl(\omega_{n-2}
                \bigl(\omega_{n-1}(\omega_n+1)+1\bigr)+1\bigr)+1\Bigr),
     \notag
\end{align}
the method coefficients are given by
\begin{align*}
    \alpha_{0,n} &= -\frac{
        \omega_{n-3}^2\bigl(
            2P_5\,\omega_{n-2}^2
            - 2(7\sqrt{5}-27)\,Q_5\,\omega_{n-2}
            - 7\sqrt{5}+27
        \bigr)}{D_5} \\[4pt]
    \alpha_{1,n} &= \frac{
        \omega_{n-3}\bigl(
            2\omega_{n-3}\,P_5\,\omega_{n-2}^2
            - 2(7\sqrt{5}-27)(\omega_{n-3}+1)\,Q_5\,\omega_{n-2}
            - (7\sqrt{5}-27)(\omega_{n-3}+2)
        \bigr) - 7\sqrt{5}+27}{D_5} \\[4pt]
    \alpha_{2,n} &= \tfrac{1}{88}(-25-9\sqrt{5}) \\[4pt]
    \alpha_{3,n} &= \tfrac{\sqrt{5}}{2} \\[4pt]
    \alpha_{4,n} &= \tfrac{1}{44}(-45-14\sqrt{5}) \\[4pt]
    \alpha_{5,n} &= 1 \\[4pt]
    \beta_{5,n}  &= -\frac{N_{\beta_5}}{D_{\beta_5}}.
\end{align*}

\section*{Acknowledgments}
The authors would like to express their gratitude to the anonymous referees for their insightful comments and careful reading of the manuscript, which  helped to improve the scope of this work.

\section*{Statements and declarations}

\noindent The authors have no competing interests that are relevant to the content of this article.

\subsubsection*{Data availability statement}
The Python and Wolfram \textit{Mathematica} code used to generate the numerical results and figures in this study is available at \url{https://github.com/tylerj32/optimized_lmms_for_stiff_problems_repository/tree/main}.

\subsubsection*{Use of artificial intelligence}


We used \textbf{Claude} to help with the literature review, including
the identification of the relevant structural-dynamics literature, and with
the detailed comparison of our results with previous work
\cite{klopfenstein1971numerical,shampine1997matlab,zhang2021improved,zhang2020astable},
including the symbolic verification of the correspondence between the
methods, their error constants, and the argument in our
Remark~\ref{rem:ZZMLMS3}. \textbf{Claude} also helped draft
Section~\ref{sec:spectral}, including the derivation of its formulas, and
\textbf{Gemini} was used to refine its wording. In
Section~\ref{sec:Brusselator}, \textbf{Claude} and \textbf{ChatGPT} were
used in tandem to help generate the Jacobian for solving the PDE as well as
the plotting code. We used \textbf{ChatGPT} and \textbf{Claude} to generate
some of the Python code used for error estimation and plotting in
Section~\ref{subsec:vanderpolvariable}. \textbf{Claude} was also used to
write and run the code for the additional numerical comparisons in
Sections~\ref{sec:Brusselator} and~\ref{subsec:vanderpolvariable} and for
the frozen-coefficient check in Section~\ref{subsec:adaptabless}. All
AI-generated code and text were subsequently tested, verified, and revised
by the authors, who take full responsibility for the content of this
paper.

\bibliographystyle{plain}
\bibliography{references}

\end{document}